\documentclass[notitlepage,11pt,reqno]{amsart}
\usepackage[foot]{amsaddr}
\usepackage{amssymb,nicefrac,bm,upgreek,mathtools,verbatim,enumerate}
\usepackage[final]{hyperref}
\usepackage[mathscr]{eucal}
\usepackage{dsfont}
\usepackage[normalem]{ulem}
\usepackage{amsopn}

\usepackage[margin=1in]{geometry}
\newcommand{\stkout}[1]{\ifmmode\text{\sout{\ensuremath{#1}}}\else\sout{#1}\fi}

\newtheorem{lemma}{Lemma}[section]
\newtheorem{theorem}{Theorem}[section]
\newtheorem{proposition}{Proposition}[section]
\newtheorem{corollary}{Corollary}[section]

\theoremstyle{definition}

\newtheorem{assumption}{Assumption}[section]

\theoremstyle{remark}
\newtheorem{remark}{Remark}[section]
\numberwithin{theorem}{section}
\numberwithin{equation}{section}

\hypersetup{
  colorlinks=true,
  citecolor=mblue,
  linkcolor=mblue,
  urlcolor = blue,
  anchorcolor = blue,
  frenchlinks=false,
  pdfborder={0 0 0},
  naturalnames=false,
  hypertexnames=false,
  breaklinks}
\usepackage[capitalize,nameinlink]{cleveref}

\usepackage[abbrev,msc-links,nobysame]{amsrefs} 

\crefname{section}{Section}{Sections}
\crefname{subsection}{Section}{Sections}
\crefname{condition}{Condition}{Conditions}
\crefname{hypothesis}{Hypothesis}{Conditions}
\crefname{assumption}{Assumption}{Assumptions}
\crefname{lemma}{Lemma}{Lemmas}
\crefname{fact}{Fact}{Facts}

\Crefname{figure}{Figure}{Figures}

\crefformat{equation}{\textup{#2(#1)#3}}
\crefrangeformat{equation}{\textup{#3(#1)#4--#5(#2)#6}}
\crefmultiformat{equation}{\textup{#2(#1)#3}}{ and \textup{#2(#1)#3}}
{, \textup{#2(#1)#3}}{, and \textup{#2(#1)#3}}
\crefrangemultiformat{equation}{\textup{#3(#1)#4--#5(#2)#6}}%
{ and \textup{#3(#1)#4--#5(#2)#6}}{, \textup{#3(#1)#4--#5(#2)#6}}%
{, and \textup{#3(#1)#4--#5(#2)#6}}

\Crefformat{equation}{#2Equation~\textup{(#1)}#3}
\Crefrangeformat{equation}{Equations~\textup{#3(#1)#4--#5(#2)#6}}
\Crefmultiformat{equation}{Equations~\textup{#2(#1)#3}}{ and \textup{#2(#1)#3}}
{, \textup{#2(#1)#3}}{, and \textup{#2(#1)#3}}
\Crefrangemultiformat{equation}{Equations~\textup{#3(#1)#4--#5(#2)#6}}%
{ and \textup{#3(#1)#4--#5(#2)#6}}{, \textup{#3(#1)#4--#5(#2)#6}}%
{, and \textup{#3(#1)#4--#5(#2)#6}}

\crefdefaultlabelformat{#2\textup{#1}#3}
\usepackage[refpage]{nomencl}
\makenomenclature

\makeatletter
\newcommand \Dotfill {\leavevmode \cleaders \hb@xt@ .6em{\hss .\hss }\hfill \kern \z@}\makeatother

\usepackage[utf8]{inputenc}
\usepackage{apptools}
\AtAppendix{\counterwithin{lemma}{section}}

\newcommand{\vertiii}[1]{{\left\vert\kern-0.25ex\left\vert\kern-0.25ex\left\vert #1 
    \right\vert\kern-0.25ex\right\vert\kern-0.25ex\right\vert}}
\newcommand{\lamstr}{\lambda^{\mspace{-2mu}*}}

\newcommand{\eom}{\mathcal{M}} 

\newcommand{\Ag}{{\mathcal{A}}}  
\newcommand{\bAg}{\bm{\mathcal{A}}}  
\newcommand{\sB}{{\mathscr{B}}}  
\newcommand{\Cc}{\mathcal{C}}   
 
\newcommand{\sF}{{\mathfrak{F}}}   

\newcommand{\cK}{{\mathcal{K}}}  
\newcommand{\cS}{{\mathcal{S}}}  
\newcommand{\cP}{{\mathcal{P}}}  

\newcommand{\cX}{{\mathcal{X}}}

\newcommand{\RR}{\mathds{R}}
\newcommand{\NN}{\mathds{N}}

\newcommand{\RN}{{\mathds{R}^{N}}}
\newcommand{\Rd}{{\mathds{R}^{d}}}
\DeclareMathOperator{\Exp}{\mathbb{E}}
\DeclareMathOperator{\Prob}{\mathbb{P}}
\newcommand{\D}{\mathrm{d}}

\newcommand{\Ind}{\mathds{1}}   
\newcommand{\Adm}{\mathfrak{U}}

\newcommand{\Sob}{{\mathscr W}}    
\newcommand{\Sobl}{{\mathscr W}_{\mathrm{loc}}} 

\newcommand{\df}{\coloneqq}

\DeclareMathOperator*{\trace}{Tr}
\DeclareMathOperator*{\dist}{dist}

\DeclareMathOperator*{\supp}{support}

\newcommand{\sorder}{{\mathfrak{o}}}
\newcommand{\grad}{\nabla}

\newcommand{\abs}[1]{\lvert#1\rvert}
\newcommand{\norm}[1]{\lVert#1\rVert}
\newcommand{\babs}[1]{\bigl\lvert#1\bigr\rvert}

\usepackage{color}
\definecolor{dmagenta}{rgb}{.4,.1,.5}
\definecolor{dblue}{rgb}{.0,.0,.5}
\definecolor{mblue}{rgb}{.0,.0,.7}
\definecolor{ddblue}{rgb}{.0,.0,.4}
\definecolor{dred}{rgb}{.7,.0,.0}
\definecolor{dgreen}{rgb}{.0,.5,.0}
\definecolor{Eeom}{rgb}{.0,.0,.5}

\newcommand{\ttl}{\Large On ergodic control problem 
for viscous Hamilton--Jacobi\\[5pt] equations for weakly coupled
elliptic systems}

\author[Ari Arapostathis]{Ari Arapostathis$^\dag$}
\address{$^{\dag}$Department of ECE,
The University of Texas at Austin,
EER~7.824, Austin, TX~~78712}
\email{ari@utexas.edu}

\author[Anup Biswas]{Anup Biswas$^\ddag$}
\address{$^\ddag$Department of Mathematics,
Indian Institute of Science Education and Research,
Dr.\ Homi Bhabha Road, Pune 411008, India}
\email{anup@iiserpune.ac.in, prasun.roychowdhury@students.iiserpune.ac.in}

\author[Prasun Roychowdhury]{Prasun Roychowdhury$^\ddag$}
\begin{document}
\title[Viscous Hamilton--Jacobi equations for elliptic systems]
{\ttl}

\begin{abstract}
In this article, we study ergodic problems in the whole space $\RN$
for weakly coupled systems of viscous Hamilton-Jacobi equations with
coercive right-hand sides. The Hamiltonians are assumed to have a 
fairly general structure, and the switching rates need not be constant. 
We prove the existence of a critical value $\lamstr$ such that the ergodic 
eigenvalue problem has a solution for every $\lambda\leq\lamstr$ and
no solution for $\lambda>\lamstr$.
Moreover,
the existence and uniqueness of non-negative solutions corresponding to
the value $\lamstr$ are also established. We also exhibit the implication
of these results to the ergodic optimal control problems of 
controlled switching diffusions.
\end{abstract}
\keywords{Elliptic systems,
viscous Hamilton--Jacobi equations, infinitesimally invariant measures,
ergodic control of switching diffusion, quasi-monotone system.}

\subjclass[2010]{35J60, 35P30, 49L25}

\maketitle

\printnomenclature

\section{Introduction}
In this article we study the existence and uniqueness of solution
$(\bm{u}, \lambda)=(u_1, u_2, \lambda)$ to the equation
\begin{equation}\label{EP}\tag{EP}
\begin{split}
-\Delta u_1(x) + H_1(x,\grad u_1(x)) + \alpha_1(x)(u_1(x)-u_2(x)) &= f_1(x)-\lambda \quad \text{in}\; \RN,
\\
-\Delta u_2(x) + H_2(x,\grad u_2(x)) + \alpha_2(x)(u_2(x)-u_1(x)) &= f_2(x)-\lambda \quad \text{in}\; \RN,
\end{split}
\end{equation}
where $H_i:\RN\times\RN\to \RR$
\nomenclature[Aa]{$H_i$}{Hamiltonians in \eqref{EP}, see also \eqref{H1}-\eqref{H2}}%
denote the Hamiltonians, and $\alpha_i:\RN\to\RR_+$ are the switching rate parameters for $i=1,2$. We make the following set of assumptions
\begin{assumption}\label{A1.1}
The functions $\alpha_i:\RN\to\RR_+$ %
\nomenclature[Ab]{$\alpha_i$}{switching rates, see \cref{A1.1}}%
are continuously differentiable 
and for some constant $\upalpha_0>0$ we have 
\begin{equation}\label{EA1.1A}
\upalpha^{-1}_0\leq \alpha_i(x)\leq \upalpha_0, \quad \sup_{x}|\grad\alpha_i(x)|\leq \upalpha_0\quad \text{for}\; i=1,2.
\end{equation}
Also, the following hold.
\begin{itemize}
\item[(A1)] There exist $\ell_i\in \Cc(\RN\times\RN)$, $\xi\mapsto\ell_i(x, \xi)$ %
\nomenclature[Ac]{$\ell_i$}{Lagrangians corresponding to $H_i$}%
strictly convex, and 
$$H_i(x, p)=\sup_{\xi\in\RN}\{\xi\cdot p-\ell_i(x, \xi)\},\quad i=1,2,$$
are the Legendre transformation of $\ell_i, i=1,2$.
Moreover, $H_i\in \Cc^1(\RN\times\RN)$ %
\nomenclature[Ad]{$\Cc^k(D)$}{space of all $k$-times continuously differentiable real-valued functions on $D$}
and 
the functions $\xi\mapsto H_i(x, \xi)$ are
strictly convex, $i=1,2$.
\item[(A2)] For some constants $\gamma_i>1, i=1,2,$ we have
	\begin{align}\label{H1}
		 C_1^{-1}|p|^{\gamma_i}-C_1\leq H_i(x,p)\leq C_1(|p|^{\gamma_i}+1),\quad (x, p)\in\RN\times\RN,
	\end{align}
\begin{align}\label{H2}
	|\grad_x H_i(x,p)|\leq C_1(1+|p|^{\gamma_i})
	\quad (x, p)\in\RN\times\RN,
\end{align}
for some positive constant $C_1$ and $i=1,2$.
\end{itemize}
\end{assumption}
Since $\xi\mapsto H_i(x, \xi)$ is convex, it follows from \eqref{H1} that
\begin{equation}\label{H2A}
|\grad_{p}H_i(x,p)|\leq \tilde{C}_1(1+|p|^{\gamma_i-1})
\quad (x, p)\in\RN\times\RN, \quad i=1,2,
\end{equation}
for some positive constant $\tilde{C}_1$. In fact, for $\abs{p}>0$ we see that
\begin{align*}
|\grad_p H_i(x, p)|=\max_{|e|=1} \grad_p H_i(x, p)\cdot e
&=\max_{|z|=|p|}\, \frac{1}{|p|}\grad_p H_i(x, p)\cdot z
\\
&\leq \frac{1}{|p|} \max_{|z|=|p|} (H_i(x, p+z)-H_i(p)),
\end{align*}
using convexity. Now \eqref{H2A} follows from \eqref{H1}.

Typical examples of $H_i$ satisfying the above assumptions would be
$$H_i(x, p)=\frac{1}{\gamma_i}\langle p, a_i(x)p\rangle^{\nicefrac{\gamma_i}{2}} + 
b_i(x)\cdot p,$$
where $a_i:\RN\to\RR^{N\times N}, b_i:\RN\to\RN$ are bounded functions with bounded derivatives, and $a_i$ are uniformly elliptic for $i=1,2$. In this case,
$$\ell_i(x, \xi)=\frac{1}{\gamma^\prime_i}\langle \xi-b_i(x) , a^{-1}_i(x)
(\xi-b_i(x))\rangle^{\nicefrac{\gamma^\prime_i}{2}}\quad
\text{where}\;\; \frac{1}{\gamma_i}+\frac{1}{\gamma'_i}=1,$$
for $i=1,2$. The source terms $f_i, i=1,2,$ are assumed to satisfy the following
\begin{assumption}\label{A1.2}
The functions $f_i:\RN\to \RR, i=1,2,$ are continuously differentiable and
for some positive constant $C_2$ we have
\begin{equation}\label{F1}
|\grad f_i(x)|\leq C_2 (1+ |f_i(x)|^{2-\frac{1}{\gamma_i}})
\quad x\in\RN,
\end{equation}
for $i=1,2$. We also assume that for some $r>0$ we have
\begin{equation}\label{F2}
[|f_i(x)|+1]^{-1} \sup_{B_r(x)}|f_i(x)|<\,C_3,\quad \text{for}\; x\in\RN,
\end{equation}
for some constant $C_3$ and $i=1,2$.
\end{assumption}
Without any loss of generality, we would assume that $r=1$.
Note that \eqref{F1}-\eqref{F2} hold if we have 
$\sup_{x\in\RN}|\grad \log f_i(x)|<\infty, i=1,2,$ and $f_1, f_2$ are positive 
outside a compact set. Some other type of examples include
$f_i(x)=|x|^{\beta_1} (2+\sin((1+|x|^2)^{\beta_2}))$ for $\beta_i>0$ and
$(\beta_1+2\beta_2-1)\frac{\gamma_i}{2\gamma_i-1}\leq\beta_1, i=1,2$. 
From \eqref{F2} we also see that
$$|f_i(x)|\leq C_3(|f_i(y)|+1) \quad \text{whenever}\; \abs{x-y}\leq 1,$$
which readily gives
\begin{equation}\label{E1.2}
|f_i(x)|\leq C_3\left(\inf_{B_1(x)}|f_i(y)| + 1\right) \quad \text{for all}\; x\in\RN.
\end{equation}
\eqref{F2} will be used to obtain certain estimate on the gradient
of $\bm{u}$ (see \cref{L2.1}).

Throughout the paper, if $\cX(\RN)$ is a  subspace of real-valued functions
on $\RN$ then we define the corresponding space
$\cX(\RN\times\{1,2\}) \df \bigl(\cX(\RN)\bigr)^2$, and endow it with the
product topology, if applicable.
Thus, a function $g\in\cX(\Rd\times\{1,2\})$ is identified with the
vector-valued function
\begin{equation}\label{E-vec}
\bm{g} \,\df\, (g_1,g_2)\in\bigl(\cX(\Rd)\bigr)^2\,,\quad\text{where\ }
f_k(\cdot)\df f(\cdot,k)\,,\quad k=1,2\,.
\end{equation}
With a slight abuse in notation we write $\bm{g}\in\cX(\RN\times\{1,2\})$.

\subsection{Background and Motivation}
The system of equations \eqref{EP} arise 
as the Hamilton-Jacobi equations (HJE) in certain ergodic control problems
of diffusions in a switching environment. To be more precise, 
consider the controlled
dynamics pair $(X, S)$ where $\{X_t\}$ denotes the continuous
part governed by a controlled diffusion
$$\D{X}_t = \bm{b}(X_t, S_t)\, \D{t} -U_t\, \D{t} + \D{W}_t,$$
where $W$ is a standard $N$-dimension Brownian motion, $U$ is 
an admissible control,
 and $\{S_t\}$ is a
two state Markov process, taking values in $\{1,2\}$,
responsible for random switching. The functions $\alpha_1, \alpha_2$ 
corresponds to the switching rates which is also allowed to be state
dependent, that is,
\[
\Prob(S_{t+\delta t}=j|S_{t}=i, X_s, S_s, s\leq t)
=\left\{
\begin{array}{lll}
\alpha_1(X_t)\delta t + \sorder(\delta t)& \text{if}\; j=2,\, i=1,
\\[2mm]
\alpha_2(X_t)\delta t + \sorder(\delta t)& \text{if}\; j=1,\, i=2.
\end{array}
\right.
\]
We consider the minimization problem
$$\lamstr=\,\inf_{U\in\Adm}\liminf_{T\to\infty}\,
\frac{1}{T}\Exp\left[\int_0^T (\bm{f}(X_t) + \bm{\ell}(X_t, S_t))\D{t}\right],$$
where $\Adm$ denotes the set of all admissible controls.
Then the HJE equation associated to this optimal control problem
is given by \eqref{EP} where
$$H_i(x, p)=-b_i(x)\cdot p + \sup_{\xi\in\RN}\{p\cdot\xi -\ell_i(x,\xi)\} \quad i=1,2.$$
For a more precise description see \cref{S-App}. 
Because of the presence of both continuous dynamics and 
discrete jumps, regime-switching systems are capable of describing complex systems and the randomness of the environment. We refer to the book of 
Yin and Zhu \cite{YZ10} for more detail on regime-switching dynamics
and its application to the theory of stochastic control. 
Note that our equations \eqref{EP} includes the stochastic LQ ergodic control problem (that is,
$\gamma_1=\gamma_2=2$) for regime-switching dynamics which 
are quite popular models
in portfolio selection problems (cf. \cite[Chapter~6]{YZ99}).
One of our main results establishes
the existence of a {\it unique} optimal stationary Markov control (see 
\cref{T2.4}) for the above optimization problem.

The ergodic control problems for scalar second order elliptic equations 
have been studied extensively by several mathematicians
 and therefore, it is almost impossible to
list all the important works in this direction. Nevertheless, we mention
some of them that, in our opinion, are milestones in this topic. Ergodic 
control problems with quadratic Hamiltonian are first studied by 
Bensoussan and Freshe \cite{BenFre-92,Ben-Fre} where the authors establish the existence
and uniqueness of unbounded solutions in $\RN$. For space-time periodic
Hamiltonians, the existence and uniqueness are considered by Barles and
Souganidis \cite{Barles-01}. Ichihara \cite{Ichihara-11,Ichihara-12,Ichihara-13b}
considers the problem for a general class of Hamiltonians and 
recurrence/transience properties of the optimal feedback controls are also
discussed. We also mention the work of Cirant \cite{Cirant-14} who
investigates the ergodic control problem in $\RN$ for a fairly
general family of Hamiltonians. 
It is shown in
\cite{Cirant-14} that the problem
in $\RN$ can be {\it approximated} by the ergodic control problems in bounded domains with Neumann boundary condition. Recently, the uniqueness of unbounded solutions for a general family of 
source terms are
established by Barles and Meireles \cite{Barles-16}, which is then further improved
by the first two authors and Caffarelli \cite{ABC-19-CPDE} in the subcritical
case.
There are also
several important works studying long-time behaviour of the solutions
to certain parabolic equations and its convergence to the solutions
to the ergodic control problems: see for instance, Barles-Souganidis \cite{Barles-01}, Fujita-Ishii-Loreti \cite{Fujita-06a}, 
Tchamba \cite{Tchamba-10}, Ichihara \cite{Ichihara-12},
Barles-Porretta-Tchamba \cite{Barles-10},
Barles-Quaas-Rodr\'{\i}guez \cite{BQR21}.

On the other hand, number of works on the ergodic control problem for second-order
weakly coupled elliptic systems are very few. All existing results in this direction
consider the domain to be a torus.
See, for instance, Cagnetti-Gomes-Mitake-Tran
\cite{CGMT15}, Ley-Nguyen \cite{LN16} and references therein. We point out that \cite{LN16,CGMT15} also study the large-time asymptotics for 
the solutions to certain systems of parabolic equations, which we do not consider in this article.
However, if one assumes the action set to be compact then similar problems have been addressed in detail, see Ghosh-Arapostathis-Marcus \cite{GAM93}, Arapostathis-Borkar-Ghosh \cite[Chapter~5]{book}. One of the main challenges in studying the
weakly coupled systems lies in establishing appropriate gradient estimates of $\bm{u}$ and
bounds on the term $|u_1-u_2|$ (see \cref{P2.1} below).

\subsection{Main results}
Our chief goal in this article is to find solutions corresponding to
the critical value $\lamstr$ defined by%
\nomenclature[Ba]{$\lamstr$}{critical value to solve \eqref{EP}, see \eqref{lam}}
\begin{equation}\label{lam}
\lamstr=\sup\{\lambda\in\RR\; :\; \exists\; \bm{u}\in \Cc^2(\RN\times\{1,2\})
\; \text{such that}\; (\bm{u}, \lambda)\; \text{is a subsolution to}\; \eqref{EP}\}.
\end{equation}
The above definition is quite standard and has been used before
by several authors \cite{Ichihara-11,Barles-10,Barles-16,Tchamba-10}.
Our first main result is the following.
\begin{theorem}\label{T1.1}
Let \cref{A1.1} hold. Assume also that $\inf_{x\in\RN} f_i(x)>-\infty$ for
$i=1,2$. Then for every $\lambda\leq \lamstr$ there exists 
$\bm{u}\in \Cc^2(\RN\times\{1,2\})$ such that $(\bm{u}, \lambda)$
solves \eqref{EP}.
\end{theorem}
For proof see \cref{T2.2} below. We should mention that the proof 
of \cref{T1.1} relies on an appropriate gradient estimate and bounds on
the quantity $|u_1-u_2|$ (see \cref{P2.1}). In fact, these estimates are crucial for most of our proofs. 

We say a function $g:\RN\to\RR$ is {\it coercive} if 
$$g(x)\to\infty,\quad \text{as}\;\; \abs{x}\to\infty.$$
Given a set $\mathcal{Y}$ and two functions $g_1, g_2:\mathcal{Y}\to\RR$,
we say $g_1\asymp g_2$ in $\mathcal{Y}$ if there exist positive constants 
$\kappa_1, \kappa_2$ satisfying
$$\kappa_1 g_1\leq g_2\leq \kappa_2 g_1 \quad \text{in}\; \mathcal{Y}.$$
Next we show that there exists a solution $\bm{u}$, bounded from below,
corresponding to the eigenvalue $\lamstr$.
\begin{theorem}\label{T1.2}
Suppose that \cref{A1.1} holds. Also, assume that $f_i, i=1,2,$ are coercive.
Then there exists a solution $(\bm{u}, \lamstr)$ to \eqref{EP} where
$\inf_{\RN} u_i>-\infty$ for $i=1,2$.
\end{theorem}

For proof see \cref{T2.3}. Our next result concerns the uniqueness of solutions.
\begin{theorem}\label{T1.3}
Let \cref{A1.1,A1.2} hold. In addition, we also assume that 
$f_1\asymp f_2$ outside a compact set, and $f_i, i=1,2,$ are coercive. Let
 $(\bm{u}, \lambda)$
and $(\bm{\tilde u}, \tilde\lambda)$ be two solutions to \eqref{EP}
with $\inf_{\RN}u_i>-\infty,\, \inf_{\RN}\tilde{u}_i>-\infty$ for $i=1,2$.
Then we must have $\lambda=\tilde{\lambda}=\lamstr$ and $u_i=\tilde{u}_i+c$
for some constant $c$ and $i=1,2$.
\end{theorem}
Proof of \cref{T1.3} follows from \cref{T2.1}.
 As can be seen from above that
 \cref{A1.2} is a bit
  stronger than the usual hypotheses used to establish
 uniqueness in the super-critical regime (that is, $\gamma_i\geq 2$)
 for scalar model (cf. \cite{Barles-16}). In the scalar case, one
 generally uses an exponential transformation together with the coercive
 property of the solutions to establish uniqueness \cite{BenFre-92,Barles-16}. Similar transformation
  does not seem to work in the present setting
 because of the presence of the coupling terms. So for the uniqueness we 
 rely on the convex analytic approach of \cite{ABC-19-CPDE} and
 the estimates in \cref{P2.1}.
Also, the condition $f_1\asymp f_2$ can be relaxed provided 
$f_i, i=1,2,$ satisfy certain polynomial growth hypothesis. See \cref{T-Ext}
for further detail.

\begin{remark}
The above results correspond to a switching Markov process
having two states, that is, 
the solution $\bm{u}$ is given by a tuple $(u_1, u_2)$ of length $2$. 
All the results of this article continue to hold for weakly coupled systems with any finite number of states, provided \cref{A1.1,A1.2} are modified accordingly. 
\end{remark}

{\it The rest of the article is organized as follows. \cref{Proofs}
contains the proofs of our main results and their implication to
the optimal control problems. The proof of \cref{P2.1}
is presented in \cref{App-P2.1}, whereas \cref{App-Exis} contains few results
about the existence of solutions in bounded domains which are used in the 
proofs of \cref{T1.1,T1.2}.}

\section{Proofs of main results}\label{Proofs}
In this section we prove \cref{T1.1,T1.2,T1.3}. We start by proving a gradient
estimate which is a key ingredient for most of the proofs below. 
\begin{proposition}\label{P2.1}
Let \cref{A1.1} hold. Let $\varepsilon\in [0, 1]$. Suppose 
$B_1\Subset B_2\Subset D$
 be two given concentric balls, centered at $z$, in $\RN$.
Consider a solution $\bm{u}\in \Cc^2(D\times\{1,2\})$ to the 
system of equations 
\begin{equation}\label{EP2.1A}
\begin{split}
-\Delta u_1(x) + H_1(x,\grad u_1) + \alpha_1(x)(u_1(x)-u_2(x))+\varepsilon u_1(x) &=\,f_{1}(x) \quad \text{in}\; D,
	\\
-\Delta u_2(x) + H_2(x,\grad u_2) + \alpha_2(x)(u_2(x)-u_1(x))+\varepsilon u_2(x) &=\,f_{2}(x) \quad \text{in}\; D.
\end{split}
\end{equation}
Then there exists a constant $C>0$, dependent only on 
$\dist(B_1,\partial B_2),\gamma_i,C_1, N$ and $\sup_{B_2}(|\alpha_i|+|\grad \alpha_i|)$ for $i=1,2$, satisfying
\begin{equation}\label{EP2.1B}
	\sup_{B_1}\{|\grad u_1|^{2\gamma_1}, |\grad u_2|^{2\gamma_2}\}\leq C\Bigl(1 + \sup_{B_2}\sum_{i=1}^2(f_i)_+^{2}
	+ \sup_{B_2}\sum_{i=1}^2|\grad f_i|^{2\gamma_i/(2\gamma_i-1)}+\sup_{B_2}\sum_{i=1}^2(\varepsilon u_i)_{-}^{2}\Bigr).
\end{equation}
Furthermore, for some positive constant $\tilde{C}$, dependent only on 
$\dist(B_1,\partial B_2),\gamma_i,C_1, N, \upalpha_0$, we have
\begin{equation}\label{EP2.1C}
|u_1(z)-u_2(z)|^2\leq 
\tilde{C}\Bigl(1+ \sup_{B_2}\sum_{i=1}^2(f_i)_+^2
+ \sup_{B_2}\sum_{i=1}^2|\grad f_i|^{2\gamma_i/(2\gamma_i-1)}
+ \sup_{B_2}\sum_{i=1}^2(\varepsilon u_i)_{-}^{2}\Bigr).
\end{equation}
\end{proposition}

The proof of this Proposition is quite long and therefore, is
deferred to \cref{App-P2.1}.

Next, we show that any solution of \eqref{EP} which is bounded from below,
is actually coercive. This lemma should be compared with 
\cite[Proposition~3.4]{Barles-16} and \cite[Lemma~2.1]{ABC-19-CPDE}.
Our proof does not  use Harnack's inequality like these previous 
works. Our proof is based on the comparison principle.

\begin{lemma}\label{L2.1}
Grant \cref{A1.1,A1.2}.
Let $\bm{u}=(u_1, u_2)$ be a non-negative solution to
\begin{equation}\label{EL2.1A}
\begin{split}
-\Delta u_1 + H_1(x, \grad u_1) + \alpha_1(x)(u_1-u_2) &=f_1\quad
\text{in}\; \RN,
\\
-\Delta u_2 +  H_2(x, \grad u_2) + \alpha_2(x)(u_2-u_1) &=f_2\quad \text{in}\; \RN.
\end{split}
\end{equation}
Also, assume that $f_i, i=1,2,$ are coercive.
Then for some positive constants $M_1, M_2$ we have
\begin{equation}\label{EL2.1B}
u_i(x)\geq M_1 [f_i(x)]^{\nicefrac{1}{\gamma_i}}-M_2\quad x\in\RN, \, i=1,2.
\end{equation}
Moreover, if $f_1\asymp f_2$ outside a compact set, then 
$\frac{1}{u_i(x)}|\grad u_i|^2\leq M_3 [f_i(x)]^{\nicefrac{1}{\gamma_i}}$
outside a compact set, for some positive constant $M_3$.
\end{lemma}

\begin{proof}
Choose $R>0$ so that $f_i(x)>1$ for $|x|\geq R$.
Fix a point $x_0\in B_{R+1}^c(0)$ and define
$$\psi_i(y)=\theta|f_i(x_0)|^{\nicefrac{1}{\gamma_i}}(1-|y-x_0|^2),$$
where $\theta>0$ is to be chosen later and $i=1,2$.
Then, using \eqref{EA1.1A}-\eqref{H1}, we have in $B_1(x_0)$
\begin{align}\label{EL2.1C}
&\Delta \psi_1(y) - H_1(y, \grad\psi_1(y)) + \alpha_1(y)(\psi_2-\psi_1) + f_1(y)\nonumber
\\
&\quad \geq \Delta \psi_1(y) - C_1|\grad \psi_1|^{\gamma_1}-C_1
 + \alpha_1(y)(\psi_2-\psi_1) + f_1(y)\nonumber
\\
&\quad \geq -2N\theta|f_1(x_0)|^{\nicefrac{1}{\gamma_1}} - {2^{\gamma_1}\theta^{\gamma_1}}C_1|f_1(x_0)||y-x_0|^{\gamma_1}-C_1
-\alpha_1(y) \theta|f_1(x_0)|^{\nicefrac{1}{\gamma_1}}+ f_1(y)\nonumber
\\
&\quad \geq  f_1(x_0)\left[-2N\theta|f_1(x_0)|^{\nicefrac{1}{\gamma_1}-1} - {2^{\gamma_1}\theta^{\gamma_1}}C_1 - C_1 (f_1(x_0))^{-1}
-\upalpha_0 \theta|f_1(x_0)|^{\nicefrac{1}{\gamma_1}-1}+ \kappa\right],
\end{align}
where 
$$\Bigl[\inf_{|x|\geq R+1}\inf_{y\in B_1(x)} f(y)\Bigr](|f(x)|+1)^{-1}
\geq \kappa>0
\quad
\text{for $R$ large enough}, \text{by}\; \eqref{E1.2}.$$
Since $f_1$ is coercive, we can choose $\theta$ small and $R$ large
so that the rhs of \eqref{EL2.1C} is positive. Similarly, we can
also show that for some small $\theta$ and large $R$
$$\Delta \psi_2(y) - H_2(y, \grad \psi_2) + \alpha_2(x)(\psi_1-\psi_2)
+ f_2(y)\geq 0 \quad \text{in}\; B_1(x_0),$$
whenever $|x_0|>R$. We can now apply comparison principle, \cref{TB.1}, in $B_1(x_0)$
 to conclude that $(u_1, u_2)\geq (\psi_1, \psi_2)$
in $B_1(x_0)$ implying
$u_i(x_0)\geq \theta [f_i(x_0)]^{\nicefrac{1}{\gamma_i}}$
for $i=1,2$ and for all $|x_0|>R$.
This gives \eqref{EL2.1B}. Again, from \eqref{F1}-\eqref{F2}
and \eqref{EP2.1B} we have
$$\max\{|Du_1(x)|^{2\gamma_1}, |Du_2(x)|^{2\gamma_2}\}\leq C (1+ |f_1(x)|^{2}+|f_2(x)|^{2}),$$
for some constant $C$ and for all $x$ outside a compact set. Since $f_1\asymp f_2$ outside a compact set, the second conclusion follows from the above display and \eqref{EL2.1B}. Hence this completes the proof.
\end{proof}

We now first establish the uniqueness and then discuss the existence
results, that is, we assume \cref{T1.1,T1.2} and prove \cref{T1.3} first,
and then we prove \cref{T1.1,T1.2}.
\subsection{Uniqueness}
We begin by introducing a few notations. By $\bm{g}=(g_1, g_2)\in 
\Cc^2(\RN\times\{1,2\})$ we mean $g_i\in \Cc^2(\RN)$ for $i=1,2$.
Define the operator $\bAg=(\Ag_1,\Ag_2):\Cc^2(\RN\times\{1,2\})\to
\Cc^2(\RN\times\RN\times\{1,2\})$ by %
\nomenclature[Bc]{$\bAg, \Ag_k$}{extended generator, see \eqref{EAg}}%
\begin{equation}\label{EAg}
\Ag_k \bm{g}(x,\xi) \,\df\, \Delta g_k(x) - \xi\cdot \grad g_k(x)
+ \alpha_k (x) \sum_{j=1}^2(g_j(x)-g_k(x)),\quad
(x,\xi)\in\RN\times\RN,\, k=1,2,
\end{equation}
with $\bm{g}=(g_1,g_2)\in \Cc^2(\RN\times\{1,2\})$.
Also, $\Cc^2_c(\RN\times\{1,2\})$ denotes the class of functions
in $\Cc^2(\RN\times\{1,2\})$ with compact support.
Let $\cP(\RN\times\RN\times\{1,2\})$ denotes the set of Borel probability measures
$\bm\mu=(\mu_1,\mu_2)$, with $\mu_i=\bm\mu(\cdot\times\{i\})$ being a 
sub-probability measure.
For a function $\bm{h}\colon\RN\times\RN\to\RR^2$
we use the notation %
\nomenclature[Bd]{$\bm\mu(\bm{h})$}{action of $\bm{\mu}$ on $\bm{h}$}%
\begin{equation*}
\bm\mu(\bm{h})
\,\df\,\int_{\RN\times\RN} \bigl\langle \bm{h}(x,\xi)\,,\bm\mu(\D{x},\D\xi)\bigr\rangle
\,=\,
\sum_{k=1}^2 \int_{\RN\times\RN} h_k(x,\xi)\mu_k(\D{x},\D\xi)\,.
\end{equation*}
We define
\begin{equation*}
\eom \,\df\, \Bigl\{\bm\mu\in\cP(\RN\times\RN\times\{1,2\})\,\colon
\bm\mu\bigl(\bAg\bm{g}\bigr)
=0\quad \forall\; \bm{g}\in\Cc^2_c(\RN\times\{1,2\})\Bigr\}\,.
\end{equation*}
Let
\begin{equation}\label{Cost}
F_k(x,\xi) \,\df\, f_k(x) + \ell_k(x, \xi)\quad k=1,2,
\end{equation}
where $\ell_k$ is given by \cref{A1.1}.
Now define
\begin{equation}\label{EMF}
\eom_{\bm{F}}\,\df\, \bigl\{\bm\mu\in\eom\,\colon \bm\mu(\bm{F})<\infty\bigr\}\,,
\end{equation}
and %
\nomenclature[Be]{$\bar\lambda$}{optimal value, see \eqref{LP}}%
\begin{equation}\label{LP}
\overline\lambda \,\df\, \inf_{\mu\in\eom}\;\bm\mu(\bm{F})
= \inf_{\mu\in\eom_{\bm{F}}}\;\bm\mu(\bm{F})\,.\tag{\sf{LP}}
\end{equation}
In \cref{L2.3} below we show that $\eom_{\bm{F}}$ is non-empty.
Our next result shows that $\lamstr$ in \eqref{lam} is smaller than $\bar\lambda$.

\begin{lemma}\label{L2.2}
Consider the setting of \cref{T1.3}. Then we must have 
$\lamstr\leq\bar\lambda$.
\end{lemma}

\begin{proof}
We only consider the case when $\bar\lambda<\infty$, otherwise there is
nothing to prove. Let $\bm\mu\in\eom$ be such that $\bm\mu(\bm{F})<\infty$.
Since $\bm\mu\in\eom$ we have
\begin{equation}\label{EL2.2A}
\bm\mu(\bAg\bm{g})=\sum_{k=1}^2\int_{\RN\times\RN}\Ag_k\bm{g}(x, \xi)
\mu_k(\D{x},\D\xi)=0\quad \text{for all}\; \bm{g}\in \Cc^2_c(\RN\times\{1,2\}).
\end{equation}
Let $\bm{u}=(u_1, u_2)$ be a non-negative solution to \eqref{EP}
corresponding to $\lamstr$,
that is,
\begin{equation}\label{EL2.2B}
\begin{split}
-\Delta u_1(x) + H_1(x,\grad u_1(x)) + \alpha_1(x)(u_1(x)-u_2(x)) &= f_1(x)-\lamstr \quad \text{in}\; \RN,
\\
-\Delta u_2(x) + H_2(x,\grad u_2(x)) + \alpha_2(x)(u_2(x)-u_1(x)) &= f_2(x)-\lamstr \quad \text{in}\; \RN.
\end{split}
\end{equation}
Existence of $u$ follows from \cref{T1.2}.
From \cref{L2.1} we also know that $u_i, i=1,2,$ are coercive. We would
modify $\bm{u}$ suitably so that it can be used in \eqref{EL2.2A} as a test 
function. To do so, we consider a family of concave functions.

For $r>0$, we let $\chi^{}_r$ be a concave function
in $\Cc^2(\RR)$ such that
$\chi^{}_r(t)= t$ for $t\le r$, and $\chi'_r(t) = 0$ for $t\ge 3r$.
Then $\chi'_r$ and $-\chi''_r$ are nonnegative, and the latter is supported
 on $[r,3r]$.
In addition, we select $\chi^{}_r$ so that
\begin{equation}\label{EL2.2C}
\abs{\chi''_r(t)} \,\le\, \frac{2}{t}\qquad\forall\; t>0\,.
\end{equation}
In particular, we may define $\chi_r$ by specifying
\[
\chi^{''}_r(t)=\left\{
\begin{array}{lll}
\frac{4}{3}\frac{r-t}{r^2} & \text{if}\; r\leq t\leq \frac{3r}{2},
\\[2mm]
-\frac{2}{3r} & \text{if}\; \frac{3r}{2}\leq t\leq \frac{5r}{2},
\\[2mm]
\frac{4}{3}(\frac{t}{r^2}-\frac{3}{r}) & \text{if}\; \frac{5r}{2}\leq t\leq 3r.
\end{array}
\right.
\]
Using \eqref{EL2.2B} we now compute
\begin{equation}\label{EL2.2D}
\begin{aligned}
&\Delta \chi^{}_r(u_k) - \xi\cdot \grad\chi^{}_r(u_k)
+\alpha_k\sum_{j=1}^2 (\chi^{}_r(u_j)-\chi_r(u_k))\\
&\,=\,
\chi''_r(u_k) \abs{ \grad{u_k}}^2
+ \chi'_r(u_k)\bigl(\Delta u_k -\xi\cdot \grad{u_k}\bigr)
+\alpha_k\sum_{j=1}^2 (\chi^{}_r(u_j)-\chi_r(u_k))\\
&\,=\, \chi''_r(u_k) \abs{ \grad{u_k}}^2
+\chi'_r(u_k)\Bigl(\lamstr 
+ H_k(x, \grad u_k) - f_k
-\xi\cdot\grad{u_k}\Bigr)\\
&\mspace{100mu} + 
\alpha_k\sum_{j=1}^2 \bigl(\chi^{}_r(u_j)-\chi_r(u_k)-
\chi'_r(u_k) (u_j-u_k)\bigr)
\\
&\,=\, \chi''_r(u_k) \abs{ D{u_k}}^2
+\chi'_r(u_k)\Bigl(\lamstr - f_k- \ell_k(x, \xi)\Bigr)\\
&\quad + \chi'_r(u_k)\Bigl(\ell_k(x, \xi) -\xi\cdot \grad{u_k}+
H_k(x, \grad u_k)\Bigr)
 +\alpha_k\sum_{j=1}^2 \bigl(\chi^{}_r(u_j)-\chi_r(u_k)-
\chi'_r(u_k) (u_j-u_k)\bigr)
\,.
\end{aligned}
\end{equation}
Thus, defining
\begin{equation*}
G_{r,k}[\bm{u}](x)\,\df\,
\alpha_k\sum_{j=1}^2 \bigl(\chi^{}_r(u_j)-\chi_r(u_k)-
\chi'_r(u_k) (u_j-u_k)\bigr)\,,
\end{equation*}
and integrating \cref{EL2.2D} with respect to a $\bm\mu$, we obtain
\begin{equation}\label{EL2.2E}
\begin{aligned}
\sum_{k=1}^n\int_{\RN\times\RN}
&\chi'_r\bigl(u_k(x)\bigr)
\Bigl(f_k(x)+ \ell_k(x,\xi)- \lamstr\Bigr)
\,\mu_k(\D{x},\D{\xi})\\
 &\,=\,
\sum_{k=1}^2\int_{\RN\times\RN}
\chi'_r\bigl(u_k(x)\bigr)\Bigl(\ell_k(x, \xi) -\xi\cdot \grad{u_k}+
H_k(x, \grad u_k)\Bigr)
\,\mu_k(\D{x},\D{\xi})\\
&\mspace{50mu}+\sum_{k=1}^2\int_{\RN\times\RN}
\Bigl(\chi''_r\bigl(u_k(x)\bigr) \babs{ D{u_k}(x)}^2+G_{r,k}[\bm{u}](x)\Bigr)
\,\mu_k(\D{x},\D{\xi})\,.
\end{aligned}
\end{equation}
Next we show that the last term on the rhs of \eqref{EL2.2E} goes to
$0$ as $r\to\infty$. Since $f_1\asymp f_2$ outside a compact set and
$\bm\mu(\bm{f})=\sum_{k=1}^2\int_{\RN\times\RN} f_k(x) \mu_k(\D{x},
\D\xi)<\infty$, we obtain
\begin{equation}\label{EL2.2F}
\int_{\RN\times\RN} (|f_1(x)|+|f_2(x)|)\mu_1(\D{x}, \D\xi)<\infty,
\quad \text{and}\quad 
\int_{\RN\times\RN} (|f_1(x)|+|f_2(x)|)\mu_2(\D{x}, \D\xi)<\infty.
\end{equation}
Therefore, using \cref{L2.1} and \eqref{EL2.2C}, we get 
\begin{align*}
\sum_{k=1}^2\int_{\RN\times\RN}
|\chi''_r\bigl(u_k(x)\bigr)| \babs{ D{u_k}(x)}^2 \mu_k(\D{x},\D\xi)
&\leq \sum_{k=1}^2\int_{\RN\times\RN} \Ind_{\{r<u_k(x)<3r\}} \frac{2}{u_k(x)} \babs{ D{u_k}(x)}^2\mu_k(\D{x},\D\xi)
\\
&\leq \kappa \sum_{k=1}^2\int_{\RN\times\RN} \Ind_{\{r<u_k(x)<3r\}} 
|f_k(x)|^{\nicefrac{1}{\gamma_i}}\mu_k(\D{x},\D\xi),
\end{align*}
for some constant $\kappa$. Since $u_k, k=1,2,$ are coercive, using
dominated convergence theorem it follows that the rhs of the above display tends to $0$ as $r\to\infty$. Again, since $\chi'\leq 1$, it follows that
$$ |G_{r,k}[\bm{u}](x)|\leq 2\upalpha_0\Ind_{A^c_r}(x) |u_1(x)-u_2(x)|
\quad \text{for all}\; x\in\RN, \; k=1,2,$$
where $A_r=\{x\; :\; u_2(x)\vee u_1(x)\leq r\}$. Using 
\eqref{F1}-\eqref{F2} and \eqref{EP2.1C} we then have
$$ |G_{r,k}[\bm{u}](x)|\leq \kappa_1\Ind_{A^c_r}(x) (|f_1(x)|+|f_2(x)|)
\quad \text{for all}\; x\in\RN, \; k=1,2,$$
for some constant $\kappa_1$. Again using \eqref{EL2.2F} and dominated
convergence theorem we thus get
$$\lim_{r\to\infty}\sum_{k=1}^2\int_{\RN\times\RN}
G_{r,k}[\bm{u}](x)\,\mu_k(\D{x},\D{\xi})=0\,.$$
From our construction, it also follows that $\chi'_{3^n}$ is an
increasing sequence. Therefore, letting $r=3^n\to\infty$ in 
\eqref{EL2.2E} and applying monotone convergence theorem we obtain
\begin{equation}\label{EL2.2G}
\bm\mu(\bm{F})-\lamstr=
\sum_{k=1}^2\int_{\RN\times\RN}
\Bigl(\ell_k(x, \xi) -\xi\cdot \grad{u_k}+H_k(x, \grad u_k)\Bigr)
\,\mu_k(\D{x},\D{\xi})\geq 0\,.
\end{equation}
Since $\bm\mu$ is arbitrary, this proves the lemma.
\end{proof}
Next we show that $\eom_{\bm{F}}$ is non-empty.
\begin{lemma}\label{L2.3}
Suppose that $\bm{u}$ is a coercive, nonnegative solution to
\eqref{EP} with eigenvalue $\lambda$. Define
$$\xi_k(x) = \grad_p H_k(x, \grad u_k(x))\quad k=1,2.$$
Then there exists a Borel probability measure $\bm\nu=(\nu_1,\nu_2)$ 
on $\RN\times\{1,2\}$
so that
$$\bm\mu_{\bm{u}}=(\mu_{1, \bm{u}}, \mu_{2, \bm{u}})\in \eom_{\bm{F}}
\quad \text{where}\quad 
\mu_{k, \bm{u}}\df \nu_k (\D{x})\delta_{\xi_k(x)}(\D\xi).$$ 
Furthermore, $\bar\lambda\leq\lambda$.
\end{lemma}

\begin{proof}
Since $H_k$ is the Fenchel–Legendre transformation of $\ell_k$, it is
well known that
\begin{equation}\label{duality}
H_k(x, p)= p\cdot\xi -\ell_k(x, \xi) \quad \text{for}\; \xi=\grad_p H_k(x, p),\end{equation}
for $k=1,2$. Therefore, we can rewrite \eqref{EP} as
\begin{equation}\label{EL2.3A}
\begin{cases}
\Delta u_1(x) - \xi_1(x)\cdot \grad u_1(x) - \alpha_1(x)(u_1(x)-u_2(x)) &= \lambda - F_1(x, \xi_1(x)) \quad \text{in}\; \RN,
\\
\Delta u_2(x) - \xi_2(x)\cdot \grad u_2(x) - \alpha_2(x)(u_2(x)-u_1(x)) &= \lambda-F_2(x, \xi_2(x)) \quad \text{in}\; \RN,
\end{cases}
\end{equation}
where $\bm{F}$ is given by \eqref{Cost}. We define the extended
generator
${\bm\Ag}_{\bm{u}}=(\Ag_{1,\bm{u}},\Ag_{2,\bm{u}}): \Cc^2(\RN\times\{1,2\})\to
\Cc^2(\RN\times\{1,2\})$ by %
\nomenclature[Ca]{${\bm\Ag}_{\bm{u}}, \Ag_{k,\bm{u}}$}{extended generator
corresponding to the solution $\bm{u}$, see \eqref{EL2.3B}}%
\begin{equation}\label{EL2.3B}
\Ag_{k,\bm{u}} \bm{g}(x) \,\df\, \Delta g_k(x) - \xi_k(x)\cdot \grad g_k(x)
+ \alpha_k (x) \sum_{j=1}^2(g_j(x)-g_k(x))\,,\quad
(x,\xi)\in\RN\times\RN,\, k=1,2.
\end{equation}
Since $\bm{u}, \bm{F}$ are coercive, there exists a switching diffusion
$(X_t, S_t)$ associated to the generator ${\bm\Ag}_{\bm{u}}$ (cf.
\cite[Chapter~5]{book}). Furthermore, the mean empirical measures
of $(X_t, S_t)$ will be tight and therefore, should have a limit point
(cf. \cite[Lemma~2.5.3]{book}).
Let $\bm\nu=(\nu_1, \nu_2)$ be one such limit points. It is also
standard to show that 
\begin{equation}\label{EL2.3C}
\sum_{k=1}^2\int_{\RN} \Ag_{k, \bm{u}}\bm{g}(x) \nu_k(\D{x})=0
\end{equation}
for all $\bm{g}\in \Cc^2_c(\RN\times\{1,2\})$. Hence it follows that
$\bm{\mu}_{\bm u}\in\eom$.

To prove the second part, we consider the concave function $\chi_r$
from \cref{L2.2}. Since $\chi_r$ is concave we have $\chi''_r\leq 0$ 
and 
$$\chi_r(u_j)-\chi_r(u_k)-\chi'_r(u_k)(u_j-u_k)\leq 0.$$
Thus, the calculation of \eqref{EL2.2D} and \eqref{duality}-\eqref{EL2.3A} gives
\begin{align*}
&\Delta \chi^{}_r(u_k) - \xi_k\cdot \grad\chi^{}_r(u_k)
+\alpha_k\sum_{j=1}^2 (\chi^{}_r(u_j)-\chi_r(u_k))
\\
&\leq \chi'_r(u_k) (\lambda-F_k(x, \xi_k(x)).
\end{align*}
Integrating both sides with $\nu_k$ and summing over $k$, we obtain
from \eqref{EL2.3C} that
$$\sum_{k=1}^2 \int_{\RN}\chi'_r(u_k) F_k(x, \xi_k(x))\nu_k(\D{x}) \leq \lambda.$$
Now letting $r\to\infty$ and using Fatou's lemma we obtain
$$\bm\mu_{\bm u}(\bm{F})\leq\lambda.$$
Thus, $\bm\mu_{\bm u}\in\eom_{\bm{F}}$ and $\bar\lambda\leq\lambda$.
\end{proof}

We note that the proof of \cref{L2.3} also works for
non-negative $\Cc^2$ super-solutions.
Combining the above result with
\cref{L2.2} we get the following corollary.

\begin{corollary}
Under the setting of \cref{T1.3} we have
\begin{equation*}
\lamstr=\inf\{\lambda\in\RR\; :\; \exists\; \text{nonnegative}\;
\bm{u}\in \Cc^2(\RN\times\{1,2\})
\; \text{such that}\; (\bm{u}, \lambda)\; \text{is a super-solution to}\; \eqref{EP}\}.
\end{equation*}
Note that the existence of a non-negative solution $\bm{u}$ for the value
$\lamstr$ follows from \cref{T1.2}.
\end{corollary}

Now we are ready to establish our uniqueness result.

\begin{theorem}\label{T2.1}
Assume the setting of \cref{T1.3}. Let $(\bm{u}, \lambda)$ be a solution
to \eqref{EP} and $\bm{u}$ is non-negative. Then
\begin{itemize}
\item[(a)] $\lambda=\lamstr=\bar\lambda=\bm\mu_{\bm{u}}(\bm{F}),$
where $\bm\mu_{\bm{u}}$ is given by \cref{L2.3}
\item[(b)] Suppose that $(\tilde{\bm u}, \tilde\lambda)$
 is another solution to \eqref{EP} and $\tilde{\bm u}$ is non-negative,
 then $\tilde\lambda=\lamstr$ and $\tilde{\bm u}=\bm{u} + c$ for some constant $c$.
\end{itemize}
\end{theorem}

\begin{proof}
(a) follows from \cref{L2.2,L2.3} and \eqref{EL2.2G}. So we consider
(b). Using \cref{L2.3}, we find a Borel probability measure 
$\tilde{\bm\nu}=(\tilde\nu_1, \tilde\nu_2)$
such that for
$$\tilde{\bm\mu}_{\tilde{\bm{u}}}=(\tilde\mu_{1, \tilde{\bm u}}, 
\tilde\mu_{2, \tilde{\bm u}}) \quad \text{with}\quad 
\tilde{\mu}_{k, \tilde{u}}\df \tilde{\nu}_k (\D{x})\delta_{\tilde\xi_k(x)}(\D\xi),\quad \tilde\xi_k(x)=\grad_p H_k(x, \grad \tilde{u}_k),$$ 
we have $\tilde\lambda=\tilde{\bm\mu}_{\tilde{\bm u}}(\bm{F})=\lamstr$.
Again, by \cite[Theorem~5.3.4]{book}, there exist strictly positive
Borel measurable functions $\bm\rho=(\rho_1, \rho_2)$ and 
$\tilde{\bm \rho}=(\tilde\rho_1, \tilde\rho_2)$ satisfying
\begin{equation}\label{ET2.1A}
\nu_k(\D{x})=\rho_k(x)\D{x}, \quad \tilde{\nu}_k(\D{x})=\tilde\rho_k(x)\D{x}
\quad \text{for}\; k=1,2.
\end{equation}
Let us now define
\begin{equation*}
\begin{gathered}
\zeta_k=\frac{\rho_k}{\rho_k + \tilde\rho_k}, \quad
\tilde\zeta_k= \frac{\tilde\rho_k}{\rho_k+\tilde\rho_k},\quad
v_k(x)=\xi_k(x)\zeta_k(x) + \tilde{\xi}_k(x)\tilde\zeta_k(x),\\
\widehat{\mu}_k(\D{x},\D{\xi})= \frac{1}{2}(\nu_k(\D{x})+\tilde\nu_k(\D{x}))
\delta_{v_k(x)}(\D{\xi})\quad \text{for}\; k=1,2.
\end{gathered}
\end{equation*}
We claim that $\widehat{\bm{\mu}}=(\widehat{\mu}_1, \widehat{\mu}_2)\in\eom$. Consider $\bm{g}=(g_1, g_2)\in C^2_c(\RN\times\{1,2\})$.
We note that 
$$\frac{1}{2}(\nu_k(\D{x})+\tilde\nu_k(\D{x}))
= \frac{1}{2}(\rho_k(x) + \tilde\rho_k(x)) \D{x}\quad \text{for}\; k=1,2.$$
A simple computation then yields
\begin{align*}
&\int_{\RN\times\RN}\Ag_k(x, \xi)\, \widehat{\mu}_k(\D{x},\D\xi)
\\
&\quad = \int_{\RN}\Bigl(
\Delta g_k(x) - v_k(x)\cdot \grad g_k(x)
+ \alpha_k (x) \sum_{j=1}^2(g_j(x)-g_k(x))\Bigr)\frac{1}{2}(\nu_1(\D{x})+\tilde\nu_1(\D{x}))
\\
&\quad = \frac{1}{2}\int_{\RN}\Bigl(
(\rho_k(x)+\tilde\rho_k(x))\Delta g_k(x) - (\xi_k(x)\rho_k(x) + \tilde{\xi}_k(x)\tilde\rho_k(x))\cdot \grad g_k(x)
\\
&\mspace{100mu}+ (\rho_k(x)+\tilde\rho_k(x))\alpha_k (x) \sum_{j=1}^2(g_j(x)-g_k(x))\Bigr)\D{x}
\\
&=\frac{1}{2}\int_{\RN}\Ag_{k, \bm{u}}\bm{g}(x)\,\nu_k(\D{x}) 
+ \frac{1}{2}\int_{\RN}\Ag_{k,\tilde{\bm u}}\bm{g}(x)\,\tilde\nu_k(\D{x}).
\end{align*}
Therefore
$$\sum_{k=1}^2\int_{\RN\times\RN}\Ag_k(x, \xi)\, \widehat{\mu}_k(\D{x},\D\xi)=\frac{1}{2}\left[\bm{\mu}_{\bm u}(\bAg_{\bm u}\bm{g})
+ \bm{\mu}_{\tilde{\bm u}}(\bAg_{\tilde{\bm u}}\bm{g})\right]=0.$$
This proves the claim. Using the convexity of $\ell_k$ in $\xi$ it is
also easily seen that $\widehat{\bm{\mu}}(\bm{F})<\infty$.  Now
from \cref{L2.2,L2.3} we see that $\bm{\mu}_{\bm u}$ and 
$\bm{\mu}_{\tilde{\bm u}}$ are optimal for \eqref{LP}. Thus we have
\begin{align*}
0&\leq \widehat{\bm{\mu}}(\bm{F})-\frac{1}{2}\bm{\mu}_{\bm u}(\bm{F})
-\frac{1}{2} \bm{\mu}_{\tilde{\bm u}}(\bm{F})
\\
&=\frac{1}{2}\sum_{k=1}^2\left[ \int_{\RN} \ell_k(x, v_k(x)) (\rho_k(x)
+\tilde{\rho}_k(x))\D{x} - \int_{\RN} \ell_k(x, \xi_k(x)) \rho_k(x)\D{x}-\int_{\RN} \ell_k(x, \tilde\xi_k(x)) \tilde\rho_k(x)\D{x}\right]
\\
&= \frac{1}{2}\sum_{k=1}^2 \left[ \int_{\RN} \bigl(\ell_k(x, v_k(x))
-\ell_k(x, \xi_k(x))\zeta_k(x) - \ell_k(x, \tilde\xi_k(x))\tilde{\zeta}_k\bigr) (\rho_k(x)+\tilde\rho_k(x))\D{x}\right]\leq 0,
\end{align*}
where the last line follows from the convexity of $\ell_k$ in $\xi$.
Therefore,
$$\sum_{k=1}^2 \left[ \int_{\RN} \Bigl(\ell_k(x, v_k(x))
-\ell_k(x, \xi_k(x))\zeta_k(x) - \ell_k(x, \tilde\xi_k(x))\tilde{\zeta}_k\Bigr) (\rho_k(x)+\tilde\rho_k(x))\D{x}\right]=0.$$
Since $\rho_k, \tilde{\rho}_k$ are strictly positive, and $\ell_k$ is strictly convex, it the follows that $\xi_k=\tilde\xi_k$ for $k=1,2$.
Since $H_k(x, \cdot)$ is strictly convex, by (A1), given $\xi$ there 
exists a unique $p$ satisfying
$$H_k(x, p) = p\cdot \xi -\ell_k(x, \xi).$$
Thus, from \eqref{duality}, we obtain $\grad u_k=\grad \tilde{u}_k$ in $\RN$,
for $k=1,2$. This, of course, implies $u_i= \tilde{u}_i + c_i$ for 
some constant $c_i$, $i=1,2$. Again, subtracting the equations of 
$\bm{u}$ from the equations of $\tilde{\bm u}$ we see that
$\alpha_1(c_1-c_2)=0$ implying $c_1=c_2$. This completes the proof.
\end{proof}

The proof of uniqueness in \cref{T2.1} requires $f_1$ to be comparable
to $f_2$ outside a compact set. This property is crucially used in \cref{L2.1,L2.2}. However, if we impose more structural assumption on
$\bm{f}$ then we could relax the requirement of $f_1\asymp f_2$.
\begin{itemize}
\item[\hypertarget{F}{{(F)}}] Suppose that there exist $\beta_1,\beta_2>1$ satisfying
$$C^{-1}_4 |x|^{\beta_i}- C_4\leq f_i(x)\leq C_4(|x|^{\beta_i}+1),\quad x\in\RN,$$
for some $C_4>0$, where 
$$\beta_2\leq \beta_1\frac{\gamma_1+1}{2},
\quad \beta_1\leq \beta_2\frac{\gamma_2+1}{2},
\quad 
\max\left\{\frac{\beta_1(\gamma_1+1)}{2\gamma_1}, \frac{\beta_2(\gamma_2+1)}{2\gamma_2}\right\}\leq \beta_1\wedge\beta_2 - 1\,.$$
\end{itemize}
As a consequence of (F) it follows that 
\begin{equation}\label{E2.22}
|f_2(x)|^{\nicefrac{2}{\gamma_1}}\leq \kappa(1+|f_1(x)|^{1+\gamma^{-1}_1})
\quad \text{and} \quad
|f_1(x)|^{\nicefrac{2}{\gamma_2}}\leq \kappa(1+|f_2(x)|^{1+\gamma^{-1}_2})
\end{equation}
for some $\kappa>0$. \cref{T2.1} can be improved as follows.
\begin{theorem}\label{T-Ext}
Suppose that \cref{A1.1}, \cref{A1.2} and \hyperlink{F}{(F)} hold.
Then the conclusions of \cref{T2.1} hold true.
\end{theorem}

\begin{proof}
We only need to modify \cref{L2.1,L2.2}. Note that \eqref{EL2.1B} holds.
Using \eqref{F1},\eqref{F2},\eqref{EP2.1B} and \eqref{E2.22} it follows that
\begin{equation}\label{T2.2A}
|\grad u_i(x)|^2\leq \kappa_1(1+|f_i(x)|^{1+\gamma^{-1}_i})
\end{equation}
for some constant $\kappa_1$. Therefore, for some compact set $\cK$ and
a constant $\kappa_3$, we obtain from \eqref{EL2.1B} that
\begin{equation}\label{T2.2B}
\frac{|\grad u_i|^2}{u_i(x)}\leq \kappa_3 |f_i(x)|\quad x\in \cK^c.
\end{equation}
Again, using \hyperlink{F}{(F)} and \eqref{T2.2A} we see that
$$|\grad u_i(x)|\leq \kappa_4 \left(1+|x|^{\frac{\beta_i(1+\gamma_i)}{2\gamma_i}}\right)\quad \text{for some}\; \kappa_4, \quad i=1,2.
$$
Using \hyperlink{F}{(F)} this also implies
\begin{equation}\label{T2.2C}
\max\{u_1(x), u_2(x)\}\leq \kappa_5 \min\{1+|f_1(x)|, 1+|f_2(x)|\}
\end{equation}
for some $\kappa_5$. Using \eqref{T2.2B} and \eqref{T2.2C} we can
complete the proof of \cref{L2.2}. Rest of the argument of \cref{T2.1}
follows without any change.
\end{proof}

\subsection{Existence} First we establish \cref{T1.1}. We see that
if $\inf_{\RN} f_i>-\infty$, then set of subsolution in \eqref{lam}
is nonempty. In particular, if we set 
$\lambda=\min_{i}\,\inf_{\RN} f_i$, then $\bm{u}=(1,1)$ is a subsolution
to \eqref{EP} with eigenvalue $\lambda$.

\begin{lemma}\label{L2.4}
Let \cref{A1.1} hold and also assume that $\bm{f}\in \Cc^1(\RN\times\{1,2\})$.
Suppose that $\bm{u}$ is a $\Cc^2$ subsolution to \eqref{EP} with some
eigenvalue $\lambda_1$. Then \eqref{EP} has a $\Cc^2$ solution for 
every $\lambda\leq\lambda_1$. 
\end{lemma}

\begin{proof}
Since $\bm{u}$ is also a subsolution for any $\lambda\leq \lambda_1$,
it is enough to show that there exists a solution $\bm{w}$ to \eqref{EP}
with eigenvalue $\lambda_1$. For a $n\in\NN$, fix $D=B_n(0)$.
Applying \cref{TB.3}, we can find a function $\bm{w}^n=(w^n_1, w^n_2)
\in \Cc^2(D\times\{1,2\})$ that satisfies
\begin{equation}\label{EL2.4A}
\begin{split}
-\Delta w^n_1(x) + H_1(x,\grad w^n_1(x)) + \alpha_1(x)(w^n_1(x)-w^n_2(x)) &= f_1(x)-\lambda_1 \quad \text{in}\; B_n(0),
\\
-\Delta w^n_2(x) + H_2(x,\grad w^n_2(x)) + \alpha_2(x)(w^n_2(x)-w^n_1(x)) &= f_2(x)-\lambda_1 \quad \text{in}\; B_n(0).
\end{split}
\end{equation}
We translate $\bm{w}^n$ to satisfy $w^n_1(0)=0$. Let $\cK$ be a compact 
subset of $\RN$. Then, by \cref{P2.1}, we get
$\sup_{n}\{|w^n_1(0)|, |w^n_2(0)|\}$ bounded and 
$$\sup_{\cK} \{|\grad w^n_1|, |\grad w^n_2|\}<C_\cK,$$
for all $n$ satisfying $B_n(0)\Supset\cK$. Thus, $\{\bm{w}^n\}$ is locally
bounded in $\Sobl^{2, p}$, uniformly in $n$. Applying a diagonalization
argument, we can find a subsequence of $\{\bm{w}^n\}$, converging to
some $\bm{w}\in \Sobl^{2, p}(\RN\times\{1,2\})$ for $p>N$. Passing limit
in \eqref{EL2.4A} gives
\begin{equation*}
\begin{split}
-\Delta w_1(x) + H_1(x,\grad w_1(x)) + \alpha_1(x)(w_1(x)-w_2(x)) &= f_1(x)-\lambda_1 \quad \text{in}\; \RN,
\\
-\Delta w_2(x) + H_2(x,\grad w_2(x)) + \alpha_2(x)(w_2(x)-w_1(x)) &= f_2(x)-\lambda_1 \quad \text{in}\; \RN.
\end{split}
\end{equation*}
We can now bootstrap the regularity of $\bm{w}$ to $\Cc^2$ using 
standard elliptic regularity theory (cf. \cite{GilTru}).
\end{proof}

Now we can complete the proof of \cref{T1.1}.
\begin{theorem}\label{T2.2}
Let \cref{A1.1} hold.
Suppose that $f_1, f_2\in \Cc^1(\RN)$ are bounded below.
Then $\lamstr$ is finite and \eqref{EP} has solution for the eigenvalue 
$\lamstr$. In particular, by \cref{L2.4}, \eqref{EP} has a solution for every $\lambda\leq\lamstr$.
\end{theorem}

\begin{proof}
From the discussion preceding \cref{L2.4} we see that 
$$\lamstr\geq\min_{i=1,2}\,\inf_{\RN}f_i\,.$$
We first show that $\lamstr<\infty$. Suppose, on the contrary, that
$\lamstr=\infty$.
Then, in view of \cref{L2.4}, there exists a sequence of solutions 
$\{(\bm\phi^k, \lambda_k)\}=\{(\phi_1^k,\phi_2^k, \lambda_k)\}$ of \eqref{EP} satisfying $\lambda_k\to\infty$, as $k\to \infty$. 
We can translate $\bm{\phi}^k$ to satisfy $\phi^k_1(0)=0$.
Since
\begin{equation}\label{ET2.2A}
\begin{split}
-\Delta \phi^k_1(x) + H_1(x,\grad \phi^k_1(x)) + 
\alpha_1(x)(\phi^k_1(x)-\phi^k_2(x)) &= f_1(x)-\lambda_k \quad \text{in}\; \RN,
\\
-\Delta \phi^k_2(x) + H_2(x,\grad \phi^k_2(x)) + 
\alpha_2(x)(\phi^k_2(x)-\phi^k_1(x)) &= f_2(x)-\lambda_k \quad \text{in}\; \RN,
\end{split}
\end{equation}
and $(f_i-\lambda_k)_+\leq (f_i)_+$ for large $k$, it follows from 
\cref{P2.1} that 
\begin{equation}\label{ET2.2B}
\sup_k\, \sup_{\cK}\{|H_1(x, \grad\phi^k_1)|, |H_2(x, \grad\phi^k_2)|\}
<\infty, \quad \sup_k\,\sup_{\cK}\{|\phi^k_1|, |\phi^k_2|\}<\infty,
\end{equation} 
for every compact set $\cK$ in $\RN$.
Setting
$$\psi_i^k\df\lambda_k^{-1}\phi^k_i\quad \text{for}\;i=1,2,$$
we see from \eqref{ET2.2A} that
\begin{equation*}
\begin{split}
-\Delta \psi^k_1(x) + \lambda_k^{-1}H_1(x,\grad {\phi}^k_1(x)) + \alpha_1(x)(\psi^k_1(x)-\psi^k_2(x)) &=\lambda_k^{-1}f_1(x)-1 \quad \text{in}\; \RN,
\\
-\Delta \psi^k_2(x) + \lambda_k^{-1}H_2(x,\grad {\phi}^k_2(x)) + \alpha_2(x)(\psi_1^k(x)-\psi^k_2(x)) &=\lambda_k^{-1}f_2(x)-1 \quad \text{in}\; \RN.
\end{split}
\end{equation*}
Using \eqref{ET2.2B} we see that $\{\bm\psi^k\}$ is locally bounded in
$\Sobl^{2, p}(\RN)$ for $p>N$. Therefore, we can find a convergence
subsequence, converging to some $\bm\psi$. \eqref{ET2.2B} also shows that
$|\grad\psi_i|=0$ implying $\bm\psi$ to be a constant. Then passing limit in
the above display we get a contradiction. Hence $\lamstr$ must be finite.

Now choose $\lambda_n<\lamstr$ such that $\lambda_n\to \lamstr$ as $n\to\infty$. Then, using \cref{L2.4}, we get a solution 
$(u_1^n,u_2^n, \lambda_n)$ to \eqref{EP}. 
Applying an argument, similar to above, we can extract a convergent
subsequence,
 converging locally to $\bm{u}=(u_1,u_2)$ and $\bm{u}$ solves
  \eqref{EP} with the eigenvalue $\lamstr$. This completes the proof.
\end{proof}

The rest of this section is devoted to the proof of \cref{T1.2}, that is,
we construct a nonnegative solution to \eqref{EP} corresponding to the
eigenvalue $\lamstr$. The broad idea of the proof is the following: We solve
the ergodic control problem \eqref{EP} on an increasing sequence of balls
$B_n$ and find solution pairs $(\bm{u}^n, \lambda_n)$ in the balls. We then 
show that $\lambda_n$ decreases to $\lambda^*$ and $\bm{u}^n\to \bm{u}$.
Using the coercivity of $\bm{f}$, we can confine the minimizer of
$\bm{u}^n$ inside a fixed compact set, independent of $n$. This also
makes $\bm{u}$ bounded from below. For this idea to work it is important that $\bm{u}^n$ attends its minimum inside $B_n$. This can be achieved if we set $\bm{u}^n=+\infty$ on $\partial B_n$. For $\gamma_i\leq 2$, this can be done using the arguments of Lasry-Lions in \cite{LL89}.
But for $\gamma_i>2$, we need to modify $\bm{f}$ to {\it attend} the boundary data.

Let $\bm{f}$ be a $\Cc^1$ function. Let $B=B_r(0)$ be the ball of radius
$r\geq 1$ around $0$.
Let $\varrho:(0, \infty)\to(0, \infty)$ be a smooth, nonnegative function satisfying
\[
\varrho(x)=\left\{
\begin{array}{lll}
	x^{-1} & \text{for}\; x\in (0, \frac{1}{2}),
	\\[2mm]
	0 & \text{for}\; x\geq 1.
\end{array}
\right.
\]
Define 
$$f_{i, \alpha}(x)= f_i(x) + [\varrho(r^2-|x|^2)]^{\alpha}
\quad x\in B,\, i=1,2,$$
for some $\alpha$ to be fixed later. Let $\beta>\max\{2, \gamma_1, \gamma_2\}$ be such that
$(\beta+1)(\gamma_i\wedge 2)>\beta+2$. Choose $\alpha>0$ to satisfy 
$\beta<\alpha<(\beta+1)(\gamma_i\wedge 2)$ for $i=1,2$. With no loss of generality,
we also assume that $1<\gamma_2\leq\gamma_1$. Our next result concerns
discounted problem in $B$.

\begin{lemma}\label{L2.5}
Let \cref{A1.1} hold. Then, for any $\varepsilon\in (0, 1)$, the system  
\begin{equation}\label{EL2.5A}
\begin{split}
-\Delta w^\varepsilon_1 + H_1(x,\grad w^\varepsilon_1) + \alpha_1(x)(w^\varepsilon_1-w^\varepsilon_2) + \varepsilon w^\varepsilon_1 &=f_{1,\alpha}\quad \text{in}\; B,
\\
-\Delta w^\varepsilon_2 + H_2(x,\grad w^\varepsilon_2) + \alpha_2(x)(w^\varepsilon_2-w^\varepsilon_1) + \varepsilon w^\varepsilon_2 &=f_{2,\alpha}\quad \text{in}\; B,
\end{split}
\end{equation}
admits a solution $(w^\varepsilon_1,w^\varepsilon_2)$ in $\Cc^2(B\times\{1,2\})$ with $w^\varepsilon_i\to \infty$ as $x\to \partial B$. Moreover, the set
$\{\varepsilon w_i^\varepsilon(0)\; :\; \varepsilon\in(0,1)\}$ is bounded
for $i=1,2$.
\end{lemma}

\begin{proof}
To find a solution to 
\eqref{EL2.5A}, first we find appropriate sub and super-solutions to \eqref{EL2.5A}.
Define $\xi^\delta(x)= -\log (r^2-\delta|x|^2)$ and let 
$(\xi^\delta_1, \xi^\delta_2)=(\kappa_1 \xi^\delta, \kappa_1 \xi^\delta)$.
It can be easily checked that, for some $\delta_0>0$ and $\delta\in (\delta_0, 1)$,
\begin{align*}
-\Delta \xi^\delta_1 + C_1(|\grad \xi^\delta_1|^{\gamma_1}+1) + \alpha_1(x)(\xi^\delta_1-\xi^\delta_2) + \varepsilon \xi^\delta_1 &\leq f_{1,\alpha}\quad \text{for}\; r-\delta_1\leq |x|<r,
\\
-\Delta \xi^\delta_2 + C_1(|\grad \xi^\delta_2|^{\gamma_2}+1) + \alpha_2(x)(\xi^\delta_2-\xi^\delta_1)
+ \varepsilon \xi^\delta_2 &\leq f_{2,\alpha}\quad \text{for}\; r-\delta_1\leq |x|<r,
\end{align*}
for some appropriate constant $\kappa_1$, dependent on $\gamma_1, \gamma_2$.
$\kappa_1, \delta_1$, and $\delta$ can be chosen independent of $\varepsilon$.
Now choose $M$ suitably large, independent of $\varepsilon, \delta$,
so that $(\kappa_1\xi^\delta_1-\frac{M}{\varepsilon}, \kappa_1\xi^\delta_2-\frac{M}{\varepsilon})$ forms a subsolution to \eqref{EL2.5A}.

Next we construct a super-solution. To this end, we consider the approximating function $\psi_n$ from \cref{LB.1}. More precisely,
we consider a sequence of functions $\bm{\psi}_n=(\psi^1_n, \psi^2_n)$
where $\psi^i_n(x)=x$ if $\gamma_i\leq 2$, otherwise $\psi^i_n=\psi_n$ 
from \cref{LB.1}.

We define $(\zeta^\delta_1, \zeta^\delta_2)=(\kappa_2\zeta, \kappa_2\zeta)$
where
$$\zeta=(r^2-\delta|x|^2)^{-\beta}\quad \text{for}\; i=1,2.$$
Using the condition $\beta<\alpha<(\beta+1)(\gamma_i\wedge 2)$,
and choosing $M$ large, independent of $n, \varepsilon, \delta$,
we see that $(\kappa_2\zeta^\delta_1+\frac{M}{\varepsilon}, \kappa_2\zeta^\delta_2+\frac{M}{\varepsilon})$ forms a supersolution
to the equation
\begin{equation*}
\begin{split}
-\Delta w^\varepsilon_1 + \psi^1_n(H_1(x,\grad w^\varepsilon_1)) + \alpha_1(x)(w^\varepsilon_1-w^\varepsilon_2) + \varepsilon w^\varepsilon_1 &=f_{1,\alpha}\quad \text{in}\; B,
\\
-\Delta w^\varepsilon_2 + \psi^2_n(H_2(x,\grad w^\varepsilon_2)) + \alpha_2(x)(w^\varepsilon_2-w^\varepsilon_1) + \varepsilon w^\varepsilon_2 &=f_{2,\alpha}\quad \text{in}\; B,
\end{split}
\end{equation*}
for all $n$. From the argument of \cref{TB.3}, we find a solution
$\bm{w}^\delta=(w^\delta_1, w^\delta_2)$ of 
\begin{equation*}
\begin{split}
-\Delta w^\delta_1 + H_1(x, \grad w^\delta_1) + \alpha_1(x)(w^\delta_1-w^\delta_2) + \varepsilon w^\delta_1 &=f_{1,\alpha}\quad \text{in}\; B,
\\
-\Delta w^\delta_2 + H_2(x, \grad w^\delta_2) + \alpha_2(x)(w^\delta_2-w^\delta_1)
+ \varepsilon w^\delta_2 &=f_{2,\alpha}\quad \text{in}\; B,
\end{split}
\end{equation*}
and  
$$\kappa_1\xi^\delta_i-\frac{M}{\varepsilon}\leq w^\delta_{i,n}\leq 
\kappa_2\zeta^\delta_i+\frac{M}{\varepsilon}\quad \text{in}\; B, \, i=1,2.$$
Using the estimates in \cref{P2.1},
we can now let $\delta\to 1$ and find a solution 
to 
\begin{equation*}
\begin{split}
-\Delta w^\varepsilon_1 + H_1(x, \grad w^\varepsilon_1) + \alpha_1(x)(w^\varepsilon_1-w^\varepsilon_2) + \varepsilon w^\varepsilon_1 &=f_{1,\alpha}\quad \text{in}\; B,
\\
-\Delta w^\varepsilon_2 + H_1(x, \grad w^\varepsilon_1) + \alpha_2(x)(w^\varepsilon_2-w^\varepsilon_1)
+ \varepsilon w^\varepsilon_2 &=f_{2,\alpha}\quad \text{in}\; B,
\end{split}
\end{equation*}
satisfying
\begin{equation}\label{EL2.5B}
-\kappa_1\log(r^2-|x|^2)-\frac{M}{\varepsilon}\leq w^\varepsilon_{i}\leq 
\kappa_2(r^2-|x|^2)^{-\beta}+\frac{M}{\varepsilon} \quad \text{in}\; B,
\, i=1,2.
\end{equation}
From \eqref{EL2.5B} we also obtain
$$\sup_{\varepsilon\in (0, 1)}\sup_{B_{1/2}(0)}|\varepsilon w^\varepsilon_i|<\infty.$$
This completes the proof.
\end{proof}

Now we can provide proof of \cref{T1.2}.
\begin{theorem}\label{T2.3}
Suppose that \cref{A1.1} holds and $f_i, i=1,2,$ are coercive. Then there exists 
a nonnegative solution to \eqref{EP} corresponding to the eigenvalue $\lamstr$.
\end{theorem}

\begin{proof}
First we find a pair $(\bm{u}^n, \lambda_n)$ solving
\begin{equation}\label{ET2.3A}
\begin{split}
-\Delta u^n_1 + H_1(x,\grad u^n_1) + \alpha_1(x)(u^n_1-u^n_2) 
 &=f^n_{1,\alpha}-\lambda_n\quad \text{in}\; B_n(0),
\\
-\Delta u^n_2 + H_2(x,\grad u^n_2) + \alpha_2(x)(u^n_2-u^n_1) &=f^n_{2,\alpha}-\lambda_n\quad \text{in}\; B_n(0),
\end{split}
\end{equation}
with $\bm{u}^n\to \infty$, as $x\to\partial B_n(0)$, where
$$f^n_{i, \alpha}= f_i + [\varrho(n^2-|x|^2)]^{\alpha},$$
and $\alpha$ is same as in \cref{L2.5}.
 Fix $n\in\NN$ and
denote by $B=B_n(0)$. Consider the solution $\bm{w}^\varepsilon$ from
\cref{L2.5}.
We set $v^\varepsilon_1=w^\varepsilon_1(x)-w^\varepsilon_1(0)$ and
$v^\varepsilon_2(x)=w^\varepsilon_2(x)-w^\varepsilon_1(0)$. From \eqref{EL2.5A}
we then find
\begin{equation}\label{ET2.3B}
\begin{split}
-\Delta v^\varepsilon_1 + H_1(x, \grad v^\varepsilon_1) + \alpha_1(x)(v^\varepsilon_1-v^\varepsilon_2) + \varepsilon w^\varepsilon_1 &=f^n_{1,\alpha}\quad \text{in}\; B,
\\
-\Delta v^\varepsilon_2 + H_2(x, \grad v^\varepsilon_2) + \alpha_2(x)(v^\varepsilon_2-v^\varepsilon_1)
+ \varepsilon w^\varepsilon_2 &=f^n_{2,\alpha}\quad \text{in}\; B.
\end{split}
\end{equation}
From our choice of $\alpha$ and \eqref{EL2.5B} we see that $f_{i, \alpha}-\varepsilon w^\varepsilon_i\geq 
\frac{1}{2}f_{i, \alpha}$ near the boundary, and since 
$\max_{B_{1/2}}\{|v^\varepsilon_1|, |v^\varepsilon_2|\}$ is bounded uniformly in $\varepsilon$ (by \cref{P2.1}),
we can see that $v^\varepsilon_i\geq \kappa_3\xi^\delta_i-M$ for some 
$\kappa_3$, using \cref{TB.1}, where $\bm{\xi}^\delta$ is same as in 
\cref{L2.5}.
 Now let
$\delta\to 1$ to get a lower bound that blows up at the boundary.
Using \cref{P2.1} and the fact $\{\varepsilon\bm{w}^\varepsilon(0)\}$
is bounded, we let
$\varepsilon\to 0$ in \eqref{ET2.3B}
to find a solution to \eqref{ET2.3A}.

Now consider the sequence of solutions $\{\bm{u}^n, \lambda_n\}$ solving
\eqref{ET2.3A}. We claim that $\lambda_n\geq\lambda_{n+1}\geq \lamstr$. 
Suppose, on the contrary, that $\lambda_n<\lambda_{n+1}$. Choose
a constant $\kappa$ so that $\bm{u}^{n+1}+\kappa$ touches $\bm{u}^n$ from below in $B_n$. This is possible as $\bm{u}^n$ blows up at the boundary.
Let $\bm{v}^n=\bm{u}^{n}-\bm{u}^{n+1}$. Also, note that
$$f^{n+1}_{i, \alpha}(x)= f_i(x)\leq f^n_{i, \alpha}\quad \text{in}\; B_n.$$
Choose $D\Subset B_n$, so that $\bm{v}^n$ vanishes at some point inside $D$.
From \eqref{ET2.3A} we then have
\begin{equation*}
\begin{cases}
-\Delta v^n_1 + h^n_1\cdot \grad v^n_1 + \alpha_1(x)(v^n_1-v^n_2) &\geq \lambda_{n+1}-\lambda_n>0\quad \text{in}\; D,
\\
-\Delta v^n_2 + h^n_2\cdot \grad v^n_1 + \alpha_2(x)(v^n_2-v^n_1) 
&\geq\lambda_{n+1}-\lambda_n>0\quad \text{in}\; D,
\end{cases}
\end{equation*}
where 
$$h^n_i(x)=\int_0^1 \grad_p H_i(x, \grad u^{n+1}_i + t(\grad u^n_i-
\grad u^{n+1}_i))\, \D{t}, \quad i=1,2.$$
By strong maximum principle we obtain $\bm{v}^n=0$ in $D$. Since $D$ is 
arbitrary, we must have $\bm{v}^n=0$ in $B_n$ which is a contradiction.
Thus we have $\lambda_n\geq \lambda_{n+1}$. An analogous argument also
shows $\lambda_n\geq\lamstr$.

Using the estimates in \cref{P2.1}, we can now find a subsequence of 
$\{\bm{u}^n\}$ converging weakly in $\Sobl^{2, p}(\RN)$ to some
$\bm{u}$. Passing limit in \eqref{ET2.3A} we see that $\bm{u}$ solves
\eqref{EP} with the eigenvalue $\lamstr$ (since $\lim_{n\to\infty}\lambda_n$ is equal to $\lamstr$). To see that $\bm{u}$ is bounded from below, we 
consider a point $(x_n, i_u)\in B_n\times\{1,2\}$ so that 
$u^n_{i_n}(x_n)$ is the minimum of $\bm{u}^n$ in $B_n$. From \eqref{ET2.3A}
we then obtain
$$\lambda_1\geq\lambda_n\geq f^n_{i_n}(x_n)\geq f_{i_n}(x_n)\geq
\min\{f_1(x_n), f_2(x_n)\}.$$
Since $f_i$ is coercive, we can find a compact set $\cK$, independent of
$n$, so that $x_n\in \cK$. Thus $\bm{u}^n\geq \min_{\cK}\{u^n_1, u^n_2\}$.
This, of course, implies that $\bm{u}$ is bounded from below. We can now
translate $\bm{u}$ to make it nonnegative. This completes the proof.
\end{proof}

We complete the section by mentioning few properties of 
$\lamstr=\lamstr(\bm{f})$.
\begin{proposition}
Let $\bm{f}, \tilde{\bm f}$ be two $C^1$ functions. Then
\begin{itemize}
\item[(i)] For any $c\in\RR$ we have $\lamstr(\bm{f}+c)=\lamstr(f)+c$.
\item[(ii)] $\bm{f}\mapsto \lamstr({\bm f})$ is concave, that is, for
$t\in[0, 1]$ we have
$$\lamstr(t\bm{f} + (1-t) \bm{f})\geq t \lamstr(\bm{f}) + (1-t)
\lamstr(\tilde{\bm f}).$$
\item[(iii)] If $\bm{f}\leq\tilde{\bm f}$, then $\lamstr(\bm{f})
\leq \lamstr(\tilde{\bm f})$. Furthermore, if we assume the setting of
\cref{T2.1} or \cref{T-Ext}, then for $\bm{f}\lneq\tilde{\bm f}$
we have $\lamstr(\bm{f}) < \lamstr(\tilde{\bm f})$.
\end{itemize}
\end{proposition}

\begin{proof}
(i) is obvious. (ii) follows from the convexity of $H_i$ and the definition
\eqref{lam}. Also, first part of (iii) follows from the definition \eqref{lam}. To Prove the second part, we suppose, on the contrary, that
$\lamstr(\bm{f}) = \lamstr(\tilde{\bm f})$. Let 
$\tilde{\bm u}$ be a non-negative solution to \eqref{EP} with right-hand
side $\tilde{\bm f}$ and eigenvalue $\lamstr(\tilde{\bm f})$. Then
$\tilde{\bm u}$ would be a supersolution to \eqref{EP} with right-hand
side $\bm{f}$. From \cref{L2.3} we know that for
$$\tilde\xi_k(x) = \grad_p H_k(x, \grad \tilde{u}_k(x))\quad k=1,2,$$
there exists a Borel probability measure $\tilde{\bm\nu}=(\tilde\nu_1,\tilde\nu_2)$ 
so that
$$\tilde{\bm\mu}_{\tilde{\bm u}}=(\tilde\mu_{1, u}, \tilde\mu_{2, u})
\quad \text{with}\quad 
\tilde\mu_{k, \tilde{\bm u}}\df \tilde\nu_k (\D{x})
\delta_{\tilde\xi_k(x)}(\D\xi)\in \eom_{\bm{F}}.
$$
Moreover, $\tilde{\bm\mu}_{\tilde{\bm u}}(\bm{F})\leq\lamstr(\bm{f})$. 
By \cref{T2.1} or \cref{T-Ext} we must have $\tilde{\bm\mu}_{\tilde{\bm u}}(\bm{F})=\lamstr(\bm{f})$. Again, using \eqref{EL2.2G}, we obtain
$$\sum_{k=1}^2\int_{\RN}
\Bigl(\ell_k(x, \tilde\xi_k(x)) -\tilde\xi_k(x)\cdot \grad{u_k}+H_k(x, \grad u_k)\Bigr)
\,\tilde\nu_k(\D{x})=0.$$
Since $\tilde\nu_k$ has strictly positive densities (cf. \cite[Theorem~5.3.4]{book}), it follows that $\grad u_k=\grad \tilde{u}_k$.
Thus $u_k=\tilde{u}_k+c_k$ for some constants $c_k$ for $k=1,2$. Subtracting
the equation satisfied by $\bm{u}$ and $\tilde{\bm u}$ we obtain
$$\alpha_1(x)(c_2-c_1)=\tilde{f}_1(x)-f_1(x),
\quad \text{and}\quad \alpha_2(x)(c_1-c_2)=\tilde{f}_2(x)-f_2(x),$$
which implies
$$\frac{\tilde{f}_1(x)-f_1(x)}{\alpha_1(x)} +
\frac{\tilde{f}_2(x)-f_2(x)}{\alpha_2(x)}=0.$$
But this is not possible as $\bm{f}\lneq \tilde{\bm f}$. Hence we must have
$\lamstr(\bm{f}) < \lamstr(\tilde{\bm f})$.
\end{proof}

\subsection{Application to optimal ergodic control}\label{S-App}
In this section, we describe the optimal ergodic control problem associated with the system of equations \eqref{EP}. Denote by $\cS=\{1,2\}$, the state space of 
the switching continuous time Markov process.
We introduce the regime switching controlled diffusion process
on a given complete probability space $(\Omega, \sF, \Prob)$.
This is a process $(X_{t}, S_{t})$ in $\RN\times \cS$
governed by the following stochastic differential equations:
\begin{equation}\label{E2.29}
\begin{aligned}
\D X_t &\,=\, \bm{b}(X_t,S_t) \D t - U_t\, \D{t} +  \D W_t\,,\\
\D S_t &\,=\, \int_{\RR} h(X_t,S_{t^-},z)\wp(\D t, \D z)\,,
\end{aligned}
\end{equation}
for $t\ge 0$, where
\begin{itemize}
\item[(i)]
$(X_0, S_{0})$ are prescribed deterministic initial data;
\item[(ii)]
$W$ is an $N$-dimensional standard Wiener process;
\item[(iii)]
$\wp(\D t,\D z)$ is a Poisson random measure on $\RR_{+}\times\RR$ with intensity
$\D t\times \mathfrak{m}(\D z)$, where $\mathfrak{m}$ is the Lebesgue measure on $\RR$;
\item[(iv)]
$\wp(\cdot,\cdot)$, $W(\cdot)$ are independent;
\item[(v)]
The function $h\colon\Rd\times\cS \times \RR \to \RR$ is defined by
\begin{equation*}
h(x,i,z)\,\df\,\begin{cases}
j - i & \text{if}\,\, z\in \Delta_{ij}(x),\\[1mm]
0 & \text{otherwise},
\end{cases}
\end{equation*}
where for $i,j\in\cS, i\neq j,$ and fixed $x$, $\Delta_{ij}(x)$ are left closed right open
disjoint intervals of $\RR$ having length $m_{ij}(x)$, and
$$m_{11}(x)=-\alpha_1(x),\; m_{12}=\alpha_1(x),\; m_{21}(x)=\alpha_2(x),\;
m_{22}(x)=-\alpha_2(x).$$
\end{itemize}
Note that $\bm{M}(x)\df(m_{ij})$
 can be interpreted as the rate matrix of the
Markov chain $S_{t}$ given that $X_t=x$. In other words, 
\begin{equation*}
\Prob(S_{t+h}=j\, |\, X_t, S_t) \,=\, \begin{cases}
m_{S_t j}(X_t)h + \sorder(h) & \text{if\ } S_t\neq j\,,
\\[2mm]
1+ m_{S_t j}(X_t)h + \sorder(h) & \text{if\ } S_t = j\,,
\end{cases}
\end{equation*}
and $X$ behaves like an ordinary diffusion process governed by \eqref{E2.29} between two consecutive 
jumps of $S$. 

We assume $\bm{b}:\RN\times\cS\to\RN$ to be a bounded $\Cc^1$ function
with bounded first derivatives. The process 
$\{U_t\}$ takes values in $\RN$ and non-anticipative in nature, that is,
the sigma fields%
\nomenclature[Da]{$\sB(\cX)$}{collection of Borel subsets of $\cX$}%
$$\sigma\{X_0, S_0, W_s, U_s, \wp(A, B)\; :\; A\in \sB([0, s]),\,
B\in\sB(\RR), \, s\leq t\},$$
and
$$\sigma\{W_s-W_t, \wp(A, B)\; :\; A\in \sB([s, \infty)),\, B\in\sB(\RR)
,\, s\geq t\},$$
are independent. To introduce the admissible class of controls we set 
$\gamma_1=\gamma_2=\gamma$ and define
$$\Adm=\left\{U\; :\; \Exp\left[\int_0^T |U_t|^{\gamma'}\D{t}\right]<\infty
\quad \text{for all}\; T>0\right\},$$
where $\gamma'$ is the H\"{o}lder conjugate of $\gamma$.
We also assume $\tilde\ell_i$ to satisfy the following bound
$$\kappa^{-1}|\xi|^{\gamma'}-\kappa\leq\tilde\ell_i(x, \xi)\leq
\kappa(1+\abs{\xi}^{\gamma'}),$$
for some $\kappa>0$ and $\xi\mapsto\ell_i(x, \xi)$ are strictly
convex, $i=1,2$. We let
$$H_i(x, p)= -b_i(x)\cdot p + \sup_{\xi\in\RN}\{p\cdot\xi-\tilde\ell_i(x, \xi)\}
\quad i=1,2.$$
Also, assume that $H_i\in \Cc^1(\RN\times\RN)$ and 
the functions $\xi\mapsto H_i(x, \xi)$ are
strictly convex for $i=1,2$.
It can be easily shown that \eqref{E2.29} has a unique
strong solution for $U\in\Adm$. Now we can state the main result of this section.

\begin{theorem}\label{T2.4}
Consider the setting of \cref{T2.1} or \cref{T-Ext}. We also assume that $\gamma_1=\gamma_2=\gamma$.
Then 
\begin{equation}\label{ET2.4A}
\inf_{U\in\Adm}\;\liminf_{T\to\infty}
\frac{1}{T}\Exp\left[\int_0^T \bigl(\bm{f}(X_t, S_t) + \tilde{\bm\ell}(X_t,S_t,U_t)\D{t}\bigr)\right]\, =\, \lamstr.
\end{equation}
Furthermore, the stationary Markov control 
$$(\grad_p H_1(x,\grad u_1(x)), \grad_p H_2(x,\grad u_2(x)))+\bm{b}$$
is optimal where $\bm{u}$ is a non-negative solution to \eqref{EP} corresponding to the eigenvalue $\lamstr$. Furthermore, from 
\eqref{EL2.2G}, we also see that this is the only optimal stationary Markov control.
\end{theorem}

\begin{proof}
We only show that the lhs of \eqref{ET2.4A} is larger than $\lamstr$.
Rest of the proof follows from \cref{T2.1} or \cref{T-Ext}. Consider $U\in\Adm$ so that 
\begin{equation}\label{ET2.4B}
\liminf_{T\to\infty}
\frac{1}{T}\Exp\left[\int_0^T \bigl(\bm{f}(X_t, S_t) + \bm{\ell}(X_t,S_t, U_t)\D{t}\bigr)\right]
=
\lim_{T_n\to\infty}
\frac{1}{T_n}\Exp\left[\int_0^{T_n} \bigl(\bm{f}(X_t, S_t) + \bm{\ell}(X_t, S_t, U_t)\D{t}\bigr)\right]<\infty.
\end{equation}
We define the mean empirical measure as on $\RN\times\RN\times\cS$ as
follows
$$\bm{\mu}^n(A_1\times A_2\times C)=\frac{1}{T_n}
\Exp\left[\int_0^{T_n} \Ind_{A_1\times C\times A_2}(X_t,S_t, U_t)\D{t}\bigr)\right], \quad A_i\in \sB(\RN), C\subset \cS.$$
From the definition of $\bm{\mu}^n$ it follows that 
$$\bm{\mu}^n(\bm{F})=\frac{1}{T_n}\Exp\left[\int_0^{T_n} \bigl(\bm{f}(X_t, S_t) + \bm{\ell}(X_t, S_t, U_t)\D{t}\bigr)\right],$$
where $\bm{F}$ is given by \eqref{Cost}. From the 
coercivity property of $\bm{F}$ it can be easily seen that $\{\bm{\mu}^n\}$
is tight. Let $\bm{\mu}$ be a sub-sequential limit of $\{\bm{\mu}^n\}$.
Using \cite[Lemma~2.5.3]{book} and the lower-semicontinuity property of
weak convergence we see that $\bm{\mu}\in\eom_{\bm F}$. Again, from
\eqref{ET2.4B}, we get
$$\liminf_{T\to\infty}
\frac{1}{T}\Exp\left[\int_0^T \bigl(\bm{f}(X_t, S_t) + \bm{\ell}(X_t,S_t, U_t)\D{t}\bigr)\right]\geq \bm{\mu}(\bm{F}).$$
By \cref{L2.2} we obtain
$$\liminf_{T\to\infty}
\frac{1}{T}\Exp\left[\int_0^T \bigl(\bm{f}(X_t, S_t) + \bm{\ell}(X_t,S_t, U_t)\D{t}\bigr)\right]\geq\lamstr.$$
This completes the proof.
\end{proof}

\appendix

\section{Proof of \cref{P2.1}}\label{App-P2.1}
Part of the proof of this Proposition is inspired from \cite{Ichihara-12}.

\begin{proof}
With no loss of generality, we assume that $z=0$, $B_1=B_1(0)$, and 
$B_2=B_2(0)$. We first show that
\begin{equation}\label{EP2.1D}
	\sup_{B_1}\{|\grad u_1|^{2\gamma_1}, |\grad u_2|^{2\gamma_2}\}\leq C\big(1+ \sup_{B_2}\sum_{i=1}^2(f_i)_+^{2}
	+ \sup_{B_2}\sum_{i=1}^2|\grad f_i|^{2\gamma_i/(2\gamma_i-1)}+ |u_1(0)-u_2(0)|^{2}+\sup_{B_2}\sum_{i=1}^2(\varepsilon u_i)_{-}^{2}\big).
\end{equation}
Let $\rho:B_2\to [0, 1]$ be smooth, radial function which is decreasing along the radius, 
	$\rho=1$ in $B_1$, and  $\supp(\rho)\subset B_2$. We take 
	$\gamma=\min\{\gamma_1,\gamma_2\}$ and define $\eta=\rho^{\frac{4\gamma}{\gamma-1}}$. Without loss of generality we may assume that 
	$$\max_{B_2}\{\eta|\grad u_1|^{2}, \eta|\grad u_2|^{2}\}=\eta(x_0)|\grad u_1(x_0)|^{2}\quad \text{for some $x_0$ in }\, B_2.$$
	Define $\theta(x)=\eta(x)|\grad u_1(x)|^{2}=\eta(x) w(x)$
	where $w(x)=|\grad u_1(x)|^{2}$. Then we have 
	$\grad \theta(x_0)=0$ and $\Delta \theta(x_0)\leq 0$. 
	We may also assume that $\theta(x_0)>1$. Otherwise, if 
	$\theta(x_0)\leq 1$, 
	we get 
	$$\max_{B_1}\{\eta|\grad u_1|^{2}, \eta|\grad u_2|^{2}\}\leq \theta(x_0)\leq 1,$$ 
	and \eqref{EP2.1D} follows. Therefore, we work with $\theta(x_0)>1$. We see that
	\begin{equation}\label{EP2.1E}
		0=\grad \theta(x_0) = \eta(x_0)\grad w(x_0) + w(x_0)\grad \eta(x_0).
	\end{equation}
	Now onward we shall evaluate everything at the point $x=x_0$, 
	without explicitly mentioning the point $x_0$. Then
	\begin{align*}
		\Delta w &= \trace[(D^2 u_1)^2] + \grad(\Delta u_1)\cdot \grad u_1
		\\
		&=\trace[(D^2 u_1)^2] + \grad(H_1(x,\grad u_1) + \alpha_1(u_1-u_2)+\varepsilon u_1 -f_1)\cdot \grad u_1
		\\
		&=\trace[(D^2 u_1)^2] +\bigg[\grad_xH_1+(\grad_{p}H_1) D^2u_1
		+ (u_1-u_2)\grad\alpha_1
		+\alpha_1(\grad u_1-\grad u_2)+\varepsilon \grad u_1 -\grad f_1\bigg]\cdot \grad u_1.
\end{align*}
Using \eqref{EP2.1E}, we then obtain
\begin{align*}
0&\geq \Delta \theta=\eta\Delta w +2\grad\eta\cdot\grad w + w\Delta \eta
	\\
&=\eta\bigg[\trace[(D^2 u_1)^2] +  \grad_xH_1\cdot\grad u_1+(-2w\eta^{-1})\grad \eta\cdot \grad_{p}H_1+ (u_1-u_2)\grad\alpha_1\cdot\grad u_1
\\
&+\alpha_1(\grad u_1-\grad u_2)\cdot \grad u_1+\varepsilon w -\grad f_1\cdot \grad u_1\bigg]-2\eta^{-1}w|\grad\eta|^2+w\Delta\eta
\\
&\geq \eta\bigg[\trace[(D^2 u_1)^2] -|\grad_xH_1||\grad u_1|- 2w\eta^{-1}|\grad_{p}H_1| |\grad \eta|+ (u_1-u_2)\grad\alpha_1\cdot\grad u_1
\\
&+\alpha_1(\grad u_1-\grad u_2)\cdot \grad u_1 -|\grad f_1||\grad u_1|\bigg]-2\eta^{-1}w|\grad\eta|^2-w|\Delta\eta|.
\end{align*}
Using \eqref{EP2.1A} , \eqref{H1} and the inequality $(t_1+t_2+t_3+t_4)^2
\geq \frac{1}{4} t_1^2 -[(t_2)^2_{-}+ (t_3)^2_{-}+(t_4)^2_{-}]$,
we get (taking $t_1=H_1+C_1\geq 0$)
\begin{equation*}
N\trace[(D^2 u_1)^2]\geq (\Delta u_1)^2 \geq \bigg(\frac{1}{4C^2_1}|\grad u_1|^{2\gamma_1} -(f_1+C_1)_+^2-\alpha_1^2(u_1-u_2)^2-(\varepsilon u_1)_{-}^2\bigg).
\end{equation*}
Since $N\geq 1$ and $\eta\leq 1$, we obtain
\begin{align}\label{EP2.1F}
\frac{1}{4NC^2_1}\eta|\grad u_1|^{2\gamma_1}
&\leq \eta\trace[(D^2 u_1)^2] +(f_1+C_1)_+^2+\eta\alpha_1^2(u_1-u_2)^2+(\varepsilon u_1)_{-}^2\nonumber
\\
&\leq (f_1 + C_1)_+^2+\eta\alpha_1^2(u_1-u_2)^2+(\varepsilon u_1)_{-}^2
+\eta|\grad_xH_1||\grad u_1|+ 2w|\grad_{p}H_1| |\grad \eta|\nonumber
 \\
&\quad -\eta(u_1-u_2)\grad\alpha_1\cdot\grad u_1
-\eta\alpha_1(\grad u_1-\grad u_2)\cdot \grad u_1\nonumber
\\
&\qquad +\eta|\grad f_1||\grad u_1|+2 \eta^{-1} w|\grad\eta|^2 +  w|\Delta\eta|.
\end{align}
We observe that
	$$\eta(x_0)\alpha_1(x_0)(|\grad{u_1}(x_0)|^2-\grad u_2(x_0)\cdot \grad u_1(x_0))
	\geq \eta(x_0)\alpha_1(x_0)(|\grad{u_1}(x_0)|^2-|\grad u_2(x_0)||\grad u_1(x_0)|)\geq 0.$$
Also, by Mean Value Theorem, there exist $\zeta\in B_2$
, with $|\zeta|<|x_0|$,
and a constant $\kappa_1>0$, dependent on  $\sup_{B_2}|\alpha_1|$,
	 such that 
\begin{align*}
\eta(x_0)\alpha_1^2(u_1(x_0)-u_2(x_0))^2
&\leq \eta(x_0) \kappa_1 \big(|\grad u_1(\zeta)-\grad u_2(\zeta)|^2 + |u_1(0)-u_2(0)|^2\big)
\\
&\leq \eta(\zeta) \kappa_1\big(|\grad u_1(\zeta)-\grad u_2(\zeta)|^2 + |u_1(0)-u_2(0)|^2\big)
\\
& \leq \kappa_1\big(4\theta(x_0)+|u_1(0)-u_2(0)|^2\big),
\end{align*}
where in the second line we use the fact that $\eta$ is radially decreasing.
Another application of the Mean Value Theorem and  a similar  estimate
as above gives us, for some $\zeta_1$ with $|\zeta_1|<|x_0|$,
\begin{align*}
-\eta(x_0)(u_1(x_0)-u_2(x_0))\grad\alpha_1(x_0)\cdot\grad u_1(x_0)
&\leq  \eta(x_0)|u_1(x_0)-u_2(x_0)||\grad \alpha_1(x_0)||\grad u_1(x_0)|
\\
& \leq \kappa_2  \sqrt{\eta(x_0)}\big(|\grad u_1(\zeta_1)|+|u_1(0)-u_2(0)|\big)
\sqrt{\theta(x_0)}
\\
& \leq \kappa_2 \big(\sqrt{\eta(\zeta_1)}|\grad u_1(\zeta_1)|+|u_1(0)-u_2(0)|\big)
\sqrt{\theta(x_0)}
\\
&\leq \kappa_2 \big(2\theta(x_0) + |u_1(0)-u_2(0)|^2\big),
\end{align*}
for some constant $\kappa_2$ dependent on $\sup_{B_2}|\grad \alpha_1|$,
where in the last part we used $ab\leq{2}^{-1}(a^2+b^2)$. Again, using \cref{H2}-\eqref{H2A} and above three estimates in \eqref{EP2.1F} we deduce
that for some constant $\kappa_3$, dependent only on the bounds of $\alpha_1$, it holds
\begin{align}\label{EP2.1G}
&\frac{1}{4N C^2_1}\eta|\grad u_1|^{2\gamma_1} \nonumber
\\
&\quad\leq 2(f_1)_+^2 + 2C^2_1 + (\varepsilon u_1)_{-}^2+ C_1\eta(1+|\grad u_1|^{\gamma_1})|\grad u_1|
+ 2\tilde{C}_1 (1+|\grad u_1|^{\gamma_1-1}) \abs{\grad u_1}^2|\grad \eta|
\nonumber
\\
&\qquad + \kappa_3\big(\eta|\grad u_1|^2+|u_1(0)-u_2(0)|^2\big)+\eta|\grad f_1||\grad u_1|+|\grad u_1|^2(2\eta^{-1}|\grad\eta|^2+ |\Delta\eta|).
\end{align}
Using Young's inequality for appropriate $\delta>0$ to $|\grad u_1||\grad f_1|$, we obtain  $\kappa_\delta>0$ satisfying
	$$|\grad u_1||\grad f_1|\leq \delta|\grad u_1|^{2\gamma_1}+ \kappa_\delta|\grad f_1|^{2\gamma_1/(2\gamma_1-1)}.$$
Since $|\grad u_1(x_0)|\geq 1$, and $\gamma_1>1$, we also have 
$$(1+|\grad u_1|^{\gamma_1})|\grad u_1|\leq 2|\grad u_1|^{\gamma_1+1},
\quad \text{and}\quad (1+|\grad u_1|^{\gamma_1-1}) \abs{\grad u_1}^2
\leq 2 \abs{\grad u_1}^{\gamma+1}.
$$
Thus, from \eqref{EP2.1G} we obtain a constant $\kappa_4>0$, dependent on $N, C_1,
\kappa_1, \kappa_2,\kappa_3.$ and $\kappa_\delta$,  such that 
\begin{align*}
\eta|\grad u_1|^{2\gamma_1} &\leq  \kappa_4\bigg(1+ (f_1)_+^2+|u_1(0)-u_2(0)|^2
+(\varepsilon u_1)_{-}^2 + |\grad f_1|^{2\gamma_1/(2\gamma_1-1)}
\\
&+|\grad u_1|^{\gamma_1+1}  |\grad \eta|+|\grad u_1|^2\big(2\eta^{-1}|\grad\eta|^2+ |\Delta\eta|\big)\bigg).
\end{align*}
Now we define $V(x_0)=\eta(x_0)|\grad u_1(x_0)|^{2\gamma_1}$ and 
$\beta= \frac{\gamma_1+1}{2\gamma_1}\in(\frac{1}{\gamma_1},1)$. Then 
\begin{align*}
\eta|\grad u_1|^{2\gamma_1}& \leq  \kappa_4 \bigg(1+ (f_1)_+^2+|u_1(0)-u_2(0)|^2+(\varepsilon u_1)_{-}^2+|\grad f_1|^{2\gamma_1/(2\gamma_1-1)}
\\
&\quad +V^\beta\eta^{-\beta} |\grad \eta|+V^{1/\gamma_1}\big(2\eta^{-(\gamma_1+1)/\gamma_1}|\grad\eta|^2+ \eta^{-1/\gamma_1}|\Delta\eta|\big)\bigg)
\\
&\leq \kappa_4\bigg(1+ (f_1)_+^2+|u_1(0)-u_2(0)|^2+(\varepsilon u_1)_{-}^2+|\grad f_1|^{2\gamma_1/(2\gamma_1-1)}\bigg)
\\
&\quad + \kappa_4 V^\beta\bigg(\eta^{-\beta} |\grad \eta|+2\eta^{-2\beta}|\grad\eta|^2+ \eta^{-\beta}|\Delta\eta|\bigg),
\end{align*}
where in the last line we used $V(x_0)\geq (\eta(x_0)|\grad u_1|^2)^{\gamma_1}>1$, $\eta\leq 1$ and $\frac{1}{\gamma_1}<\beta$. 
To conclude the proof of \eqref{EP2.1D} it is enough to show that 
$\eta^{-\beta} |\grad \eta|$ and $\eta^{-\beta}|\Delta\eta|$ are bounded quantities. Recall that $\eta=\rho^\tau$ where $\tau={\frac{4\gamma}{\gamma-1}}$ with $\gamma=\min\{\gamma_1,\gamma_2\}$. It is easily seen that 
	$\tau= \max\{\frac{4\gamma_1}{\gamma_1-1}, \frac{4\gamma_2}{\gamma_2-1}\}$.
	A simple calculation yields
	\begin{align*}
		\eta^{-\beta} |\grad \eta|&=\tau\rho^{\tau-1-\tau\beta}|\grad\rho|,
		\\
		\eta^{-\beta}|\Delta\eta| &\leq\tau\{\rho^{\tau-1-\tau\beta}|\Delta\rho|+(\tau-1)\rho^{\tau-2-\tau\beta}|\grad\rho|^2\}.
	\end{align*}
	We observe that $1-\beta=\frac{\gamma_1-1}{2\gamma_1}$, and thus,
\begin{align*}
		\tau(1-\beta)-1\geq \frac{\gamma_1-1}{2\gamma_1}{\frac{4\gamma_1}{\gamma_1-1}}-1=1,\quad 
		\text{ and }\quad  \tau(1-\beta)-2\geq 0.
\end{align*}
	Hence, there exist constant $C>0$ satisfying 
\begin{align*}
		\eta(x_0) |\grad u_1|^{2\gamma_1}\leq C\bigg(1+ (f_1)_+^2+|u_1(0)-u_2(0)|^2+(\varepsilon u_1)_{-}^2+|\grad f_1|^{2\gamma_1/(2\gamma_1-1)}\bigg).
\end{align*}
Now taking supremum over $B_2$, we can write
\begin{equation*}
		\sup_{B_1}\{|\grad u_1|^{2\gamma_1},|\grad u_2|^{2\gamma_2}\}\leq C\big(1+ \sup_{B_2}(f_1)_+^2
		+ \sup_{B_2}|\grad f_1|^{2\gamma_1/(2\gamma_1-1)}+ |u_1(0)-u_2(0)|^2+\sup_{B_2}(\varepsilon u_1)_{-}^2\big).
\end{equation*}
If the maximum is attained at the second component we can repeat an 
analogous argument.  This gives us \cref{EP2.1D}.

Next, we prove \eqref{EP2.1C}.
Suppose, on the contrary, that there exists 
$\{(u^n_i, f^n_i, \alpha^n_i, \varepsilon_n)\}_n$ with $\alpha^n_i$ satisfying \eqref{EA1.1A}, and
\begin{equation}\label{EP2.1H}
\begin{cases}
-\Delta u^n_1(x) + H_1(x,\grad u^n_1) + \alpha^n_1(x)(u^n_1(x)-u^n_2(x))+\varepsilon u^n_1(x) &=\,f^n_{1}(x) \quad \text{in}\; D,
\\
-\Delta u^n_2(x) + H_2(x,\grad u^n_2) + \alpha^n_2(x)(u^n_2(x)-u^n_1(x))+\varepsilon_n u^n_2(x) &=\,f^n_{2}(x) \quad \text{in}\; D,
\end{cases}
\end{equation}
 and 
\begin{equation}\label{EP2.1I}
|u^n_1(0)-u^n_2(0)|^2 > n \bigl(1+ \sup_{B_2}\sum_{i=1}^2(f_i^n)_+^2
		+ \sup_{B_2}\sum_{i=1}^2|\grad f_i^n|^{2\gamma_i/(2\gamma_i-1)}+\sup_{B_2}\sum_{i=1}^2(\varepsilon u_i)_{-}^{2}\bigr).
	\end{equation}
First of all note that we can always set $u^n_1(0)=0$. Therefore,
by \eqref{EP2.1I}, we see that
$|u^n_2(0)|\to \infty$. Suppose that there is a subsequence,
denoted by the actual sequence, along which
 $u^n_2(0)\to \infty$.
	Define $v^n_i =\frac{1}{u^n_2(0)} u^n_i$.
Since $a^2\leq \kappa_i + a^{2\gamma_i}$ for some $\kappa_i$, for all $a\geq 0$, using \eqref{EP2.1D} and \eqref{EP2.1I} we find that 
$$\sup_{B_1}\{|\grad{v^n_1}|^{2\gamma_1},  |\grad{v^n_2}|^{2\gamma_2}\}< C \quad \text{ for all }\; n.$$
Since $(v^n_1(0), v^n_2(0))=(0, 1)$, from above estimate if follows that
$\sup_{B_1}(|v^n_1| + |v^n_2|)$ uniformly bounded in $n$. Using \eqref{H1}
and \eqref{EP2.1I} we also get
\begin{equation}\label{EP2.1J}
\sup_n\; \sup_{B_1} [\frac{1}{u_2(0)} |H_1(x, \grad u_1)|+
\frac{1}{u_2(0)} |H_2(x, \grad u_1)|]< \widehat{C}.
\end{equation}
Therefore, it follows from \eqref{EP2.1H} that
$\norm{v^n_1}_{\Sob^{2, p}(B_{\frac{1}{2}})}$, and $\norm{v^n_2}_{\Sob^{2, p}(B_{\frac{1}{2}})}$ are uniformly bounded in $n$ (cf. \cite[Theorem~9.11]{GilTru}) for any $p>N$,
and hence we can extract a weakly convergence subsequence converging to some 
$v=(v_1, v_2)\in \Sob^{2, p}(B_{\frac{1}{2}})\times \Sob^{2, p}(B_{\frac{1}{2}})$.
From the Sobolev embedding we also see that $v^n_2\to v_2$ in $C^{1, \alpha}(B_{\frac{1}{2}})$. Since $|\grad v_i^n| \to |\grad v_i| $ in $B_{\frac{1}{2}}$ and $\sup_{n}\, \sup_{B_{\frac{1}{2}}}\frac{1}{|u^n_2(0)|}|\grad u^n_i|^{\gamma_i}$
is bounded, by \eqref{H1} and \eqref{EP2.1J}, it follows that $\grad v_i=0$
in $B_{\frac{1}{2}}$. Thus, $v=(0, 1)$ in $B_{\frac{1}{2}}$. 
Now from the second equation of \eqref{EP2.1H} we get
$$-\Delta v^n_2 + \alpha^n_2 (v^n_2-v^n_1)= \frac{1}{u^n_2(0)} f^n_2 - 
	\frac{1}{ u^n_2(0)}H_2(x,\grad u^n_2)\leq
	\frac{1}{u^n_2(0)} f^n_2 + \frac{C_1}{u^n_2(0)},
$$ 
by  \eqref{H1}.  Let $\varphi$ be a nonzero, non-negative test function
supported in $B_{\frac{1}{2}}$. Multiplying the above equation by
$\varphi$, integrating over $B_{\frac{1}{2}}$ and letting $n\to\infty$
we obtain
\begin{align*}
\upalpha^{-1}_0\int_{B_{\frac{1}{2}}} \varphi(x)\D{x} 
&\leq \liminf_{n\to \infty} \int_{B_{\frac{1}{2}}}\alpha^n_2(x) v^n_2(x) \varphi(x)\D{x}
\\
&\leq \liminf_{n\to \infty} \int_{B_{\frac{1}{2}}}\varphi\big[\Delta v^n_2 + \frac{1}{u^n_2(0)} f^n_2+\alpha_2^n v_1^n + \frac{C_1}{u^n_2(0)}\big]\D{x}=0,
\end{align*}
where we use the fact that  $\sup_{B_{1/2}}|\alpha_2^n v_1^n|\leq
\upalpha_0\sup_{B_{1/2}}|v^n_1|\to 0$. Thus we arrive at a contradiction.

A similar contradiction is also arrived is $u^n_2(0)\to -\infty$ along
some subsequence. This establishes \eqref{EP2.1C}.

Finally \eqref{EP2.1B} follows from \eqref{EP2.1C} and \eqref{EP2.1D}. This completes the proof.
\end{proof}
\section{Existence results in bounded domains}\label{App-Exis}
By $D$ we denote a bounded $\Cc^{2, \delta}$ domain in $\RN$
for some $\delta>0$.
\begin{theorem}[Comparison principle]\label{TB.1}
Let $H_i\in \Cc^{1}(\RN\times\RN), i=1,2$ be given functions. Let 
$\bm{u}=(u_1, u_2)\in \Cc^2(D\times\{1, 2\})\cap \Cc^1(\bar{D}\times\{1, 2\})$
be a subsolution to
\begin{equation}\label{ETB.1}
\begin{split}
-\Delta u_1 + H(x, \grad u_1) + \alpha_1(x)(u_1-u_2)&= f_1\quad \text{in}\; D,
\\
-\Delta u_2 + H(x, \grad u_2) + \alpha_2(x)(u_2-u_1) &= f_2\quad \text{in}\; D,
\end{split}
\end{equation}
and $\bm{v}=(v_1, v_2)\in \Cc^2(D\times\{1, 2\})\cap \Cc^1(\bar{D}\times\{1, 2\})$
be a supersolution to \eqref{ETB.1}. Moreover, assume that $\bm{v}\geq \bm{u}$ on $\partial D$. Then we have $\bm{v}\geq \bm{u}$ in $\bar{D}$.
\end{theorem}

\begin{proof}
Write $w_i=v_i-u_i$. Then it follows from \eqref{ETB.1} that
\begin{equation*}
\begin{split}
-\Delta w_1 + h_1(x)\cdot \grad w_1 + \alpha_1(x)(w_1-w_2)&\geq 0\quad \text{in}\; D,
\\
-\Delta w_2 + h_2(x)\cdot \grad w_2 + \alpha_2(x)(w_2-w_1) &\geq 0\quad \text{in}\; D,
\end{split}
\end{equation*}
where 
$$ h_i(x) =\int_0^1 \grad_p H_i(x, \grad u_i(x) + t (\grad v_i(x)-
\grad u_i(x))) \D{t}, \quad i=1,2.$$ 
The result follows by applying the maximum principle, 
Busca-Sirakov \cite[Theorem~3.1]{BS04}, Sirakov \cite[Theorem~1]{Sirakov}.
\end{proof}

We next recall an existence result from \cite{AC78}. Let 
${K}_i:\bar{D}\times\RN\to\RR, i=1,2,$ be two continuous functions satisfying
$$|K_i(x, \xi)|\leq \kappa (1+|\xi|^2)\quad \text{for all}\; 
(x, \xi)\in \bar{D}\times\RN, \; i=1,2,$$
for some constant $\kappa$. We also assume that $\xi\mapsto K_i(x, \xi)$
is continuously differentiable.
\begin{theorem}\label{TB.2}
Let $\bar{\bm v}, \underline{\bm v}\in \Cc^2(\bar{D}\times\{1, 2\})$ be 
respectively a subsolution and supersolution to
\begin{align*}
-\Delta u_1 + K_1(x, \grad u_1) + \alpha_1(u_1-u_2)&=0\quad \text{in}\; D,
\\
-\Delta u_2 + K_2(x, \grad u_2) + \alpha_1(u_2-u_1)&=0\quad \text{in}\; D,
\\
u_1, u_2 &=0 \quad \text{on}\; \partial D.
\end{align*}
Also, assume that $\underline{\bm v}\leq \bar{\bm v}$ in $D$. Then there
exists a solution $\bm{u}\in \Sob^{2,p}(D\times\{1,2\})\cap \Cc(\bar{D}\times\{1,2\})$ of the above equations satisfying 
$\underline{\bm v}\leq\bm{u}\leq \bar{\bm v}$.
\end{theorem}
\begin{proof}
This can be established by mimicking the arguments of Amann-Crandall
\cite[Theorem~1]{AC78}.
\end{proof}
Note that \cref{TB.2} can be applied to find the solution for our
model provided the Hamiltonian has at-most quadratic growth in the gradient.
To apply the theorem for a general Hamiltonian we need to introduce
certain approximations.

\begin{lemma}\label{LB.1}
Suppose that $\gamma>2$. Given $C_1>0$, there exists a sequence of increasing $\Cc^{1,1}$ functions $\psi_n:[-C_1, \infty)\to [-C_1, \infty)$ satisfying the following
\begin{itemize}
\item[(i)] $\psi_n(x)\leq x$ for all $x\geq -C_1$,
\item[(ii)] $\psi_n(x)\geq \eta_1 x^{\frac{2}{\gamma}}-\eta_2$,
\item[(iii)] $0\leq \psi^\prime_n(x)\leq 1$,
\end{itemize}
where $\eta_1, \eta_2$ are positive constants independent of $n$. Furthermore,
	$$\sup_{x}\frac{\psi_n(x)}{1+|x|^2}<\infty,$$
and $\psi_n(x)\to x$ as $n\to\infty$, uniformly on compact sets.
\end{lemma}

\begin{proof}
	Define for each $n\in \mathbb{N},$
	\[
	\psi_n(x)=\left\{
	\begin{array}{lll}
		x & \text{for}\; x\leq n,
		\\[2mm]
		n-\frac{\gamma}{2}+\frac{\gamma}{2}\big(x-n+1\big)^{\frac{2}{\gamma}}& \text{for}\; x>n.
	\end{array}
	\right.
	\]
    Differentiating $\psi_n$ we get that
	\[
	\psi^\prime_n(x)=\left\{
	\begin{array}{lll}
		 1 & \text{for}\; x\leq n,
		\\[2mm]
		 \big(x-n+1\big)^{\frac{2}{\gamma}-1} & \text{for}\; x> n.
	\end{array}
	\right.
	\]
(i) and (iii) are obvious. To see (ii), we note that 
$\psi_n(x)\geq x^{\frac{2}{\gamma}}-(1+C_1^{\frac{2}{\gamma}}+C_1)$ for $x\in[-C_1, n]$. For $x>n$ we also note that 
\begin{align*}
n-\frac{\gamma}{2}+\frac{\gamma}{2}\big(x-n+1\big)^{\frac{2}{\gamma}}
&\geq (n-1)^{\frac{2}{\gamma}}+ \big(x-n+1\big)^{\frac{2}{\gamma}}
-\frac{\gamma}{2}
\\
&\geq x^\frac{\gamma}{2} -\frac{\gamma}{2}.
\end{align*}
This gives us (ii).
\end{proof}

We also require the following gradient estimate which follows by
repeating the arguments in the proof of \cref{P2.1}.
\begin{lemma}\label{LB2}
Grant \cref{A1.1}.
Let $\epsilon\in [0,1)$ and $f_1,f_2\in \Cc^1(\Rd)$. 
Let $\bm{u}$ be a $\Cc^2$ function satisfying 
\begin{align*}
-\Delta u_1(x) + \psi_n^1(H_1(x,\grad u_1)) + \alpha_1(x)(u_1(x)-u_2(x))+\varepsilon u_1(x) &=f_1(x)\quad \text{ in}\; \bar{B}_2,
\\
-\Delta u_2(x) +\psi^2_n(H_2(x,\grad u_2)) + \alpha_2(x)(u_2(x)-u_1(x)) + \varepsilon u_2(x) &=f_2(x)\quad \text{ in}\; \bar{B}_2,
\end{align*}
where $\psi^i_n$ is the approximating sequence in \cref{LB.1} if 
$\gamma_i>2$, otherwise $\psi^i_n(x)=x$. Suppose that $B_1\Subset B_2$
and $B_1, B_2$ are concentric.
Then there exists a constant $C>0$, dependent on $\dist(B_1,\partial B_2),\gamma_i$, $d,\eta_1, \eta_2,$ and $\upalpha_0$ but not on $n$ and $\bm{u}$,
 satisfying 
\begin{align*}
&\sup_{B_1}\{[\psi^1_n(H_1(x,\grad u_1))]^{2}, 
[\psi^2_n(H_2(x,\grad u_2))]^{2}\}
\\
&\quad \leq C\big(1+\sup_{B_2}\sum_{i=1}^2(f_i)_+^2
		+ \sup_{B_2}\sum_{i=1}^2|\grad f_i|^{4/3}+ |u_1(0)-u_2(0)|^2+\sup_{B_2}\sum_{i=1}^2(\varepsilon u_i)_{-}^2\big).
	\end{align*}
\end{lemma}

Now we can prove our existence result.
\begin{theorem}\label{TB.3}
Grant \cref{A1.1}.
Suppose $\varepsilon\in[0, 1]$ and $\bm{f}=(f_1, f_2)\in \Cc^1(\bar{D}\times\{1,2\})$. Let $\underline{\bm{v}}\in \Cc^2(\bar{D}\times\{1,2\})$
be a subsolution to 
\begin{equation}\label{ETB.3A}
\begin{split}
-\Delta u_1 + H_1(x, \grad u_1) + \alpha_1(x) (u_1- u_2) + \varepsilon u_1&= f_1
\quad \text{in} \; D,
\\
-\Delta u_2 + H_1(x, \grad u_2) + \alpha_2(x) (u_2- u_1) + \varepsilon u_2&= f_2
\quad \text{in} \; D.
\end{split}
\end{equation}
There there exists a solution $\bm{u}\in \Cc^2(D\times\{1, 2\})$
to \eqref{ETB.3A} satisfying $\bm{u}\geq \underline{\bm{v}}$ in $D$. 
\end{theorem}

\begin{proof}
The main idea of the proof is to use the existence result from 
\cref{TB.2} by making use of the approximation sequence in \cref{LB.1}.
A similar method was also used by Lions in \cite{Lions80} for scalar equations. In fact, the method of Lions uses more sophisticated tools 
like the Bony maximum principle to obtain an up to the boundary bounds of the
gradient. We do not use such results.
We split the proof in to two steps.

{\bf Step 1.} Fix $n\geq 1$ and consider the system of equations
\begin{equation}\label{ETB.3B}
\begin{split}
-\Delta w_1 + \psi^1_{n}(H_1(x, \grad w_1)) + \alpha_1(x) (w_1- w_2)
+ \varepsilon w_1 &= f_1
\quad \text{in} \; D,
\\
-\Delta w_2 + \psi^2_n(H_1(x, \grad w_2)) + \alpha_2(x) (w_2- w_1)
+ \varepsilon w_2 &= f_2
\quad \text{in} \; D,
\end{split}
\end{equation}
where $\psi^i_n$ is the approximating sequence from \cref{LB.1} if 
$\gamma_i>2$, otherwise $\psi^i_n(x)=x$. By \cref{LB.1}(i), we note that
$\underline{\bm{v}}$ is a subsolution to \eqref{ETB.3B}. So to apply
\cref{TB.2} we need to find a super-solution. Denote by 
$M=\max_{\partial D}\{\underline{v}_1, \underline{v}_2\}$. Let 
$\bar{\bm v}\in \Cc^2(\bar{D}\times\{1,2\})$ be the unique solution to
\begin{equation}\label{ETB.3C}
\begin{split}
-\Delta \bar{v}_1  + \alpha_1(x) (\bar{v}_1- \bar{v}_2)
+ \varepsilon \bar{v}_1 &= f_1 + \eta_2\wedge C_1
\quad \text{in} \; D,
\\
-\Delta \bar{v}_2  + \alpha_2(x) (\bar{v}_2- \bar{v}_1)
+ \varepsilon \bar{v}_2 &= f_2 + \eta_2\wedge C_1
\quad \text{in} \; D,
\\
\bar{v}_1, \bar{v}_2&= M\quad \text{on}\; \partial D,
\end{split}
\end{equation}
where $\eta_2$ is given by \cref{LB.1}(ii). In fact, using 
Sweers \cite[Theorem~1.1]{Sweer}, we can find a unique solution of \eqref{ETB.3C}
in $\Sobl^{2, p}(D)\times \Cc(\bar{D})$ and then using a standard bootstrapping
argument we can improve the regularity. Using \cref{LB.1}(ii) and \eqref{H1}
 we then obtain
from \eqref{ETB.3C} that
\begin{equation*}
\begin{split}
-\Delta \bar{v}_1 +\psi^1_{n}(H_1(x, \grad \bar{v}_1)) + \alpha_1(x) (\bar{v}_1- \bar{v}_2)
+ \varepsilon \bar{v}_1 &\geq f_1 
\quad \text{in} \; D,
\\
-\Delta \bar{v}_2 +\psi^2_{n}(H_2(x, \grad \bar{v}_2)) + \alpha_2(x) (\bar{v}_2- \bar{v}_1)
+ \varepsilon \bar{v}_2 &\geq f_2 
\quad \text{in} \; D,
\\
\bar{v}_1, \bar{v}_2&= M\quad \text{on}\; \partial D.
\end{split}
\end{equation*}
This gives us the super-solution. By \cref{TB.1} we also have
$\underline{\bm v}\leq \bar{\bm v}$ in $\bar{D}$. Now we can
apply \cref{TB.2} to find a solution 
$\bm{w}^n=(w^n_1, w^n_2)\in \Cc^2(D\times\{1,2\})\cap \Cc(\bar{D}\times\{1,2\})$
to \eqref{ETB.3B} satisfying $\underline{\bm v}\leq \bm{w}^n\leq \bar{\bm v}$ in $\bar{D}$ for all $n$. It should also be noted that $\bar{\bm v}$ is
independent of $n$.

{\bf Step 2.} We now pass to the limit in \eqref{ETB.3B} with the help
of the gradient estimate in \cref{LB2}. From step 1 we notice that
$\sup_{D}|w^n_1-w^n_2|<\infty$ uniformly in $n$. Thus, for any compact
$\cK\subset D$ we have $\max_{\cK}\{|\grad w^n_1|, |\grad w^n_2|\}<\infty$
uniformly in $n$, by \cref{LB2}. Using \eqref{ETB.3B} and
standard elliptic estimates, we get 
$$\sup_n\left\{\norm{w^n_1}_{\Sob^{2, p}(\cK)}, \norm{w^n_2}_{\Sob^{2, p}(\cK)}\right\}<\infty\quad \text{for every compact}\; \cK\subset D.$$
Using a standard diagonalization argument we can find a subsequence,
denoted by the actual one, so that $w^n_i\to u_i$ in $\Sobl^{2,p}(D)$
for $p>N$ and $w^n_i\to u_i$ in $\Cc^1_{\rm loc}(D)$, as $n\to\infty$.
Thus passing to the limit in \eqref{ETB.3B} we obtain
\begin{equation*}
\begin{split}
-\Delta u_1 + H_1(x, \grad u_1) + \alpha_1(x) (u_1- u_2)
+ \varepsilon u_1 &= f_1
\quad \text{in} \; D,
\\
-\Delta u_2 + H_1(x, \grad u_2) + \alpha_2(x) (u_2- u_1)
+ \varepsilon u_2 &= f_2
\quad \text{in} \; D,
\end{split}
\end{equation*}
and $\underline{\bm v}\leq \bm{u}\leq \bar{\bm v}$ in $D$. Moreover,
using standard theory of elliptic pde we obtain $\bm{u}\in \Cc^2(D\times\{1, 2\})$. This completes the proof.
\end{proof}

\subsection*{Acknowledgment}
We are very grateful to the referee for the careful reading of our manuscript.
The research of Anup Biswas was supported in part by a Swarnajayanti fellowship and a DST-SERB grant MTR/2018/000028. Prasun Roychowdhury was supported in part by the Council of Scientific \& Industrial Research (File no. 09/936(0182)/2017-EMR-I).

\bibliography{System-Viscous}

\end{document}